\documentclass[11pt]{amsart}
\usepackage{a4wide,pdfsync,hyperref}
\usepackage{amsmath,amssymb,esint}
\usepackage{enumerate}
\usepackage{color}

\numberwithin{equation}{section}

\allowdisplaybreaks

\newtheorem{theorem}{Theorem}[section]
\newtheorem{lemma}[theorem]{Lemma}

\newtheorem{proposition}[theorem]{Proposition}
\newtheorem{remark}[theorem]{Remark}

\newcommand{\eps}{\varepsilon}

\newcommand{\beqq}{\begin{eqnarray}}
\newcommand{\enqq}{\end{eqnarray}}

\newcommand{\enn}{\end{equation}}
\newcommand{\bef}{\begin{proof}}
\newcommand{\enf}{\end{proof}}

\let\al=\alpha

\let\f=\frac

\let\om=\omega

\let\na=\nabla

\let\pa=\partial

\def\dv{\mbox{div}}

\def\curl{\mathop{\rm curl}\nolimits}

\newcommand{\beq}{\begin{equation}}
\newcommand{\eeq}{\end{equation}}
\newcommand{\ben}{\begin{eqnarray}}
\newcommand{\een}{\end{eqnarray}}
\newcommand{\beno}{\begin{eqnarray*}}
\newcommand{\eeno}{\end{eqnarray*}}

\begin{document}
\title[The interaction between rough vortex patch  and boundary laye]{The interaction between rough vortex patch  and boundary layer}

\author[J. Huang]{Jingchi Huang}
\address{ School of Mathematics\\ Sun Yat-sen University\\ Guangzhou Guangdong 510275, China. } \email{huangjch25@mail.sysu.edu.cn}

\author[C. Wang]{Chao Wang}
\address{School of Mathematical Sciences\\ Peking University\\ Beijing 100871, China}
\email{wangchao@math.pku.edu.cn}

\author[J. Yue]{Jingchao Yue}
\address{School of Mathematical Sciences\\ Peking University\\ Beijing 100871, China}
\email{wasakarumi@163.com}

\author[Z. Zhang]{Zhifei Zhang}
\address{School of Mathematical Sciences\\ Peking University\\ Beijing 100871, China}
\email{zfzhang@pku.edu.cn}

\maketitle

\begin{abstract}
In this paper, we investigate the asymptotic behavior of solutions to the Navier-Stokes equations in the half-plane under high Reynolds number conditions, where the initial vorticity belongs to the Yudovich class and is supported away from the boundary. We establish the $L^p$ ($2\leq p< \infty$) convergence of solutions from the Navier-Stokes equations to those of the Euler equations. One of the main difficulties stems from the limited regularity of the initial data, which hinders the derivation of an asymptotic expansion. To overcome this challenge, we first prove a Kato-type criterion adapted to the Yudovich class setting.  We then obtain uniform estimates for the Navier-Stokes equations -- a non-trivial task due to the strong boundary layer effects. A key component of our approach is the introduction of a suitable functional framework, which enables us to control the interaction between the rough vortex patch and the boundary layer.
\end{abstract}

\section{Introduction}

In this paper, we study the Navier-Stokes equations at high Reynolds numbers in the domain $\mathbb R^2_+$:
\begin{align}\label{eq: NS}
	\left\{
	\begin{aligned}
		\pa_t U-\nu\Delta U+U\cdot\nabla U
		+\nabla p&=0,\\
		\operatorname{div}U&=0,\\
		U|_{t=0}&=U_0,
	\end{aligned}
	\right.
\end{align}
with  non-slip boundary condition
\begin{equation}\label{bcns}
U|_{y=0}=0.
\end{equation}
Here $U=(u,v)$ and $p$ denote the fluid velocity and the pressure respectively, and $R_e=\f{1}{\nu}$ is the Reynolds number. 

In this paper, we focus on initial data that consists of a rough vortex patch. Our main interest is to understand how the interaction between a rough patch and a boundary layer affects the behavior of the solution in the high Reynolds number regime. This constitutes a key step toward understanding the interaction between vortices and the boundary layer -- a topic of great practical interest, as exemplified by the ground effect for airplanes flying near the ground.  

Let us begin with a review of existing results in this area. In the absence of the boundary,  Constantin and Wu \cite{CW,CW1} showed that  for a vortex patch type initial data,
\beno
\|U-U^e\|_{L^2}\leq C\nu^{1/2},
\eeno
where $U^e=(u^e,v^e)$  is a solution of the Euler equations 
\begin{align}\label{eq: Euler}
	\left\{
	\begin{aligned}
		\pa_t U^e+U^e\cdot\nabla U^e+\nabla p^e&=0,\\
		\operatorname{div}U^e&=0,\\
		U^e|_{t=0}&=U_0.
	\end{aligned}
	\right.
\end{align}
Later, Abidi and Danchin \cite{AD} derived the optimal rate $\nu^{3/2}$ in $L^2.$ Sueur \cite{Sueur} provided an asymptotic expansion of the solution in the vanishing viscosity limit for fluids with vorticity exhibiting sharp variations. This asymptotic expansion was subsequently justified by Liao, Sueur, and Zhang \cite{LSZ}. Recently, for general Yudovich-type initial data, Constantin, Drivas, and Elgindi \cite{CDE} proved that the vorticity $\omega=\curl U$ satisfies
\begin{align*}
	\lim_{\nu\rightarrow 0} \|\curl U-\curl U^e\|_{L^p} =0, \quad p\in [1,\infty).
\end{align*}
When the initial vorticity possesses additional regularity($\omega_0\in L^\infty\cap B^{s}_{2,\infty}$), they further established a convergence rate dependent on this extra regularity:
\beno
\sup_{t\in[0, T] }\| \curl U-\curl U^e\|_{L^p} \leq \nu^{\f{C^2 s}{p(1+Cs)}}.
\eeno
The proof in \cite{CDE} heavily relies on the uniform bound of $\|\curl U\|_{L^p}$. In the presence of a boundary, obtaining the $L^p$ bound of $\curl U$ 
 is considerably challenging due to the boundary layer effect—even for smooth initial data. For cases with higher singularity than Yudovich-type data, specifically Dirac-type initial data (i.e., point vortices), Gallay \cite{G} demonstrated that the vorticity of the Navier-Stokes equations converges weakly to the sum of point vortices. The centers of these point vortices evolve in accordance with the Helmholtz-Kirchhoff point-vortex system. Nguyen and Nguyen \cite{NN1} later examined the interaction between a point vortex and a smooth vortex patch.

We point out that the above results focus on domains without boundaries. The situation changes significantly for domains with boundaries, owing to the presence of a boundary layer. Let us review some results on the vanishing viscosity limit in the half-plane with no-slip boundary conditions. The primary interest lies in justifying the so-called Prandtl boundary layer expansion:
\begin{align}\label{ansatz}
	\left\{
	\begin{aligned}
		&u (t,x,y)=u^e(t,x,y)+u^p(t,x,\frac{y}{\nu^{1/2}})+O(\nu^{1/2}),\\
		&v (t,x,y)=v^e(t,x,y)+\nu^{1/2}v^p(t,x,\frac{y}{\nu^{1/2}})+O(\nu^{1/2}),
	\end{aligned}
	\right.
\end{align}
where $(u^e,v^e)$ denotes the solution of the Euler equations, and $(u^p, v^p)$ denotes the solution of the Prandtl equation. In the analytic setting, the justification of this expansion has been proven in \cite{NN,SC1,WWZ}. Maekawa \cite{Maekawa} justified the expansion for cases where the initial vorticity is supported away from the boundary; see \cite{FTZ} for the three-dimensional case. This also explains why we assume the initial vortex patch does not touch the boundary. For further reference, see the insightful papers \cite{KVW, KVW1}, where it suffices to assume the initial data is analytic near the boundary. We also note the work \cite{JW} by Jiu and Wang, in which they justified the inviscid limit in the energy norm with a convergence rate $\nu^{\frac34-}$ for the Navier-slip boundary condition when the initial data is a vortex patch.  Recently, for the non-slip boundary condition and smooth patches, the last three authors \cite{WYZ} leveraged analyticity near the boundary and tangential Sobolev smoothness near the patch to establish the inviscid limit, with a convergence rate of $\nu^{\frac12 (1+\frac1p)}$.

  \subsection{Main results}

The primary objective of this paper is to establish the inviscid limit for the system \eqref{eq: NS}-\eqref{bcns} with initial data in the Yudovich class, thereby extending the results of \cite{CDE} to the half-plane. Our main result is stated as follows.

 \begin{theorem}\label{thm: Lp convergence}
Assume that the initial vorticity $\omega_0\in L_c^{\infty}(\mathbb R^2_+)$ and  $\operatorname{supp}\omega_0\subseteq\{20\leq y\leq30\}$. Then there exist a time $T_0>0,$ and two positive constants $C, C'$, (independent of $\nu$) such that for $2\le p<\infty$,
\beno
\|U(t)-U^e(t)\|_{L^p} \leq C\nu^{\f{1}{4p}-C't},\qquad  t\in[0, T_0],
\eeno
where $U^e$ is the solution of  \eqref{eq: Euler} with the boundary condition $v^{e}=0$ on $y=0$.
\end{theorem}

Let us provide some comments on our result. 

\begin{itemize}

\item  In fact, it is enough to assume that $\operatorname{supp}\omega_0\subseteq\{a\leq y\leq b\}$ for some $0<a<b<\infty$.

\item Since the initial data are prescribed in terms of vorticity, the initial velocity may not satisfy the no-slip boundary condition. Consequently, an initial layer emerges. The existence and uniqueness of solutions to \eqref{eq: NS} with such incompatible initial data have been established in \cite{Ken}.

\item The initial data considered in this paper generalize the vortex patch data studied in \cite{WYZ} and exhibit lower regularity than those in \cite{Maekawa}. In \cite{WYZ, Maekawa}, the regularity or special structure of the initial data enables the derivation of an asymptotic expansion, which effectively reduces the problem to a linear system. In our setting, however, the limited regularity precludes such an expansion. Consequently, we must establish uniform estimates directly for the Navier-Stokes equations—a fully nonlinear system. This task is further complicated by the presence of strong boundary layers, rendering the derivation of uniform estimates particularly challenging.

\item If we introduce the Prandtl boundary layer corrector $(u^p, v^p)$ by solving  
\begin{align}\label{eq: Prandtl}
	\left\{
	\begin{aligned}
		&\pa_t u^p-\pa_Y^2 u^p
		+u^p\pa_x u^e(t,x,0)
		+\big(u^p+u^e(t,x,0)\big)\pa_x u^p\\
		&\quad\quad\quad+\big(v^p-\int_0^{+\infty} \pa_x u^p(t,x,Y')dY'
		+Y\pa_y v^e(t,x,0)\big)\pa_Y u^p=0,\\
        &\pa_x u^p+\pa_Y v^p=0,\\
		&u^p|_{t=0}=0,\\
		&u^p|_{Y=0}=-u^e(t,x,0),\quad\quad\lim_{Y\rightarrow+\infty}u^p(t,x,Y)=0,
	\end{aligned}
	\right.
\end{align}
then we can obtain the $L^\infty-$convergence in the sense that
\begin{align*}
		\lim_{\nu\rightarrow0}\sup_{[0,T_0]}\left\|\Big(u-u^p(t,x,\frac{y}{\nu^{1/2}})-u^e, v-v^e\Big)\right\|_{L^\infty}=0.
	\end{align*}

\end{itemize}

 \subsection{Outline of  the proof}
Now let's give a sketch of  the proof of Theorem \ref{thm: Lp convergence}.  By H\"older inequality, we have
\beno
\|U-U^e\|_{L^p}\leq \|U-U^e\|^{\f2p}_{L^2}\|U-U^e\|^{1-\f2p}_{L^\infty}.
\eeno
Thus, it suffices to prove a quantitative convergence rate for $\|U-U^e\|_{L^2}$ and a uniform bound for $\|U\|_{L^\infty}$.  
To obtain the convergence rate of $\|U-U^e\|_{L^2}$, a natural approach is to use the Kato criterion \cite{Kato}, which reduces to verifying the following condition: 
\ben\label{eq: Kato}
\lim_{\nu\to 0}\nu \int_0^T\int_{y\leq \nu} |\na U|^2 dxdydt=0.
\een
However, in our case, $\om^e\in L^\infty$ does not guarantee $\na U^e\in L^\infty$—a condition required in the proof of the classical Kato criterion. For our purposes, we therefore need to introduce a quantitative Kato criterion tailored to our problem.
 \begin{theorem}\label{thm: kato}
Under the same assumptions on the initial data as in Theorem \ref{thm: Lp convergence},  if the vorticity $\om=\curl U$ satisfies 		
\begin{align*}
\nu^{\f12} \|\om\|_{L^2(0, T; L^2)} \leq C \nu^{\al}, 
			\end{align*}
for $0<\al\leq \f14$, then there exists $T_0>0$ such that for $ t\in [0,T_0]$, 
	\begin{align*}
			 \|U(t)-U^e(t)\|_{L^2}\leq C  \nu^{\f{\al}{2}-C' t},
			\end{align*}
where $C$ and $C'$ are constants independent of  $\nu$. 	
		\end{theorem}	 	

We believe that this new criterion is of independent interest, and its proof is partially motivated by \cite{CDE}. Based on this criterion, it suffices to prove the following key proposition.

\begin{proposition}\label{prop: uniform for U omega NS}
	Under the same assumptions on the initial data as in Theorem  \ref{thm: Lp convergence}, there exists $T_0>0$ independent of $\nu$ such that
			\begin{align*}
				 \sup_{t\in [0,T_0]}\|U(t) \|_{L^\infty}+\nu^{\f14}\|\om\|_{L^2(0,T_0; L^2)} \leq C,
			\end{align*}
			where the constant $C$ depends on $\om_0$.
\end{proposition}

The proof of Proposition \ref{prop: uniform for U omega NS}—detailed in Section 3—poses significant challenges, primarily due to the presence of two distinct layers: the boundary layer and the initial layer. To address the initial layer, we construct an initial layer corrector (see \eqref{initial layer corrector vorticity}). For the boundary layer, the conventional asymptotic expansion approach— which simplifies the nonlinear problem to a linear one—relies on high regularity of the initial data, a condition that is not satisfied in our setting. Consequently, we must handle the full nonlinearity of the system directly.

Based on Proposition \ref{prop: uniform for U omega NS} and Theorem \ref{thm: kato}, Theorem \ref{thm: Lp convergence} can be derived using the following argument
\beno
\|U-U^e\|_{L^p}\leq \|U-U^e\|^{\f2p}_{L^2}\|U-U^e\|^{1-\f2p}_{L^\infty} \leq C\nu^{\f{1}{4p}-C't}.
\eeno

\section{Kato type criterion}

In this section, we prove Theorem \ref{thm: kato}.

\begin{proof}
We focus on the time interval $[0, T_e]$, where $T_e$ is defined in Proposition \ref{Euler estimates}. 
We introduce $A$ as follows 
			\begin{align*}
				A(t,x,y)= 
				\begin{pmatrix}
					0 & yu^e(t,x,0)\\
				-yu^e(t,x,0) & 0\end{pmatrix},
			\end{align*}
which satisfies
		\begin{align*}
				\dv A|_{y=0}=U^e|_{y=0},\quad A|_{y=0}=0.
			\end{align*}
Let $z(y):=\chi(\frac{y}{\nu})$ and $U_s:=\dv(zA)=z\dv A+A\cdot\nabla z$, where the smooth cut-off function $\chi:\mathbb R_+\rightarrow [0,1]$ is defined by 
\begin{align}\label{def:chi}
	\chi(y)=\left\{
	\begin{aligned}
		&1, \quad y\leq 2,\\
	&0, \quad y\geq 3.
	\end{aligned}
	\right.
\end{align}
 Thus, $\operatorname{supp}U_s\subseteq\{0\leq y\leq3\nu\}$ near the boundary and $U_s|_{y=0}=U^e|_{y=0}$. The fact that $A$ is skew-symmetric implies $\dv U_s=0$. A direct computation gives
			\begin{align}\label{est of Uc}
				\|U_s\|_{L^2}
				+\|\pa_t U_s\|_{L^2}
				\leq C\nu^{1/2},\qquad
				\|\nabla U_s\|_{L^2}\leq C\nu^{-1/2},\qquad
				\|y^2\nabla U_s\|_{L^\infty}\leq C\nu.
			\end{align}
			
Now the energy method yields
\beno
\|U(t)\|_{L^2}\leq \|U_0\|_{L^2}=\|U^e(t)\|_{L^2}.
\eeno
Then we obtain
			\begin{align*}
				\|U(t)-U^e(t)\|_{L^2}^2
				&=\|U(t)\|_{L^2}^2
				+\|U^e(t)\|_{L^2}^2
				-2\langle U(t), U^e(t)\rangle\\
				&\leq 2\|U_0\|_{L^2}^2
				-2\langle U(t), U^e(t)-U_s(t)\rangle
				-2\langle U(t), U_s(t)\rangle:=I_1+I_2+I_3.
			\end{align*}
			
			By \eqref{est of Uc}, we have
			\begin{align*}
				|I_3|\leq 2\|U(t)\|_{L^2}\|U_s(t)\|_{L^2}
				\leq C\nu^{1/2}.
			\end{align*}
			By \eqref{eq: NS} and \eqref{eq: Euler}, we find
			\begin{align*}
				I_2=&
				\int_0^t \Big(-2\langle U\otimes U, \nabla(U^e-U_s)\rangle
				+2\nu \langle\nabla U, \nabla(U^e-U_s)\rangle
				-2\langle U,\pa_t(U^e-U_s)\rangle\Big)ds\\
				&-2\langle U_0, U_0-U_s(0)\rangle.
			\end{align*}
			Thanks to the following identities
			\begin{align*}
				-2\langle U,\pa_t(U^e-U_s)\rangle=2\langle U, U^e\cdot\nabla U^e\rangle
				+2\langle U,\pa_t U_s\rangle,
			\end{align*}
			\begin{align*}
				\langle U\otimes U, \nabla U^e\rangle
				-\langle U,U^e\cdot\nabla U^e\rangle
				=\langle (U-U^e)\otimes(U-U^e),\nabla U^e\rangle,
			\end{align*}
			we have
			\begin{align*}
				I_1+I_2&=2\langle U_0,U_s(0)\rangle
				+2\int_0^t \langle U\otimes U,\nabla U_s\rangle ds
				+2\nu \int_0^t \langle \nabla U,\nabla(U^e-U_s)\rangle ds
				\\
				&\quad+2\int_0^t \langle U,\pa_t U_s\rangle ds
				-2\int_0^t \langle(U-U^e)\otimes (U-U^e), \nabla U^e\rangle ds
				=\sum_{1\leq i\leq5} J_i.
			\end{align*}
			
			By \eqref{est of Uc} again, we have
			\begin{align*}
				|J_1|+|J_4|
				\leq 2\|U_0\|_{L^2}\|U_s(0)\|_{L^2}
				+2\int_0^t \|U(s)\|_{L^2}\|\pa_t U_s(s)\|_{L^2}ds
				\leq C\nu^{1/2}.
			\end{align*}
			By Hardy inequality and \eqref{est of Uc}, we get
			\begin{align*}
				|J_2|+|J_3|
				&\leq 2\int_0^t \left|\langle \frac{U}{y}\otimes \frac{U}{y}, y^2\nabla U_s\rangle\right|ds
				+2\nu\int_0^t \left|\langle\nabla U,\nabla(U^e-U_s)\rangle\right|ds\\
				&\leq C\int_0^t \|\nabla U\|_{L^2}^2\|y^2\nabla U_s\|_{L^\infty}ds
				+2\nu\int_0^t \|\nabla U\|_{L^2}\big(\|\nabla U^e\|_{L^2}
				+\|\nabla U_s\|_{L^2}\big)ds\\
				&\leq C\nu\|\omega\|^2_{L^2(L^2)}
				+C\nu^{1/2}\|\omega\|_{L^2(L^2)}
\leq C\nu^{\al}.
			\end{align*}
			
			For $J_5$, if we suppose $\text{supp} \omega_0\subseteq[-A,A]\times[20,30]$, then for $t$ small $\text{supp} \omega^e(t)\subseteq[-2A,2A]\times[10,40]$.
		We first define
			\begin{align*}
				B_1:=\{(x,y)\in [-3A, 3A] \times [0,50]:  |U-U^e|\geq\nu^{-1/2}\},\quad
				B_2:=[-3A, 3A]\times [0,50]\backslash B_1.
			\end{align*}
			Thus, the energy estimate implies 
			\begin{align}\label{measure est B1}
				|B_1|\leq \nu \|U-U^e\|_{L^2}^2
				\leq C\nu.
			\end{align}
			And
			\begin{align*}
				J_5
				\leq 2\Big(\int_0^t \int_{\mathbb R^2_+\backslash  [-3A, 3A]\times [0, 50] }
				+\int_0^t\int_{B_1}
				+\int_0^t\int_{B_2}\Big)
				|\nabla U^e||U-U^e|^2dxdy
				=J_{51}+J_{52}+J_{53}.
			\end{align*}
			
			To handle $J_{51}$, we first give the following Biot-Savart law which recovers the velocity from the vorticity in $\mathbb R^2_+$:
			\begin{align}\label{BS law integral}
	&U^e(x,y)=\nabla^\perp\Delta_D^{-1}\omega^e\\
	\nonumber
	&=\frac{1}{2\pi}\int_{\mathbb R^2_+}\big(-\frac{y-\tilde y}{(x-\tilde x)^2+(y-\tilde y)^2},\frac{x-\tilde x}{(x-\tilde x)^2+(y-\tilde y)^2}\big)\omega^e(\tilde x,\tilde y)d\tilde xd\tilde y\\
	\nonumber
	&\quad-\frac{1}{2\pi}\int_{\mathbb R^2_+}\big(-\frac{y+\tilde y}{(x-\tilde x)^2+(y+\tilde y)^2},\frac{x-\tilde x}{(x-\tilde x)^2+(y+\tilde y)^2}\big)\omega^e(\tilde x,\tilde y)d\tilde xd\tilde y,
\end{align}
	which implies that $\|\nabla U^e\|_{L^\infty(\mathbb R^2_+\backslash  [-3A, 3A]\times [0, 50])}\leq C$ for $t$ small, thus,
			\begin{align*}
				J_{51}\leq C\int_0^t\|U-U^e\|_{L^2}^2ds.
			\end{align*}
			By Gagliardo-Nirenberg inequality and \eqref{measure est B1},  we infer that for $r<\infty$,
			\begin{align*}
				J_{52}
				&\leq C\int_0^t \|\nabla U^e\|_{L^2(B_1)}\|U-U^e\|_{L^4}^2ds
				\leq C|B_1|^{\f12-\frac{1}{r}}\int_0^t\|\nabla U^e\|_{L^r}\|U-U^e\|_{L^2}\|\nabla(U-U^e)\|_{L^2}ds\\
				&\leq C\nu^{\f12-\frac{1}{r}}\int_0^t \|\omega^e\|_{L^r}\|U-U^e\|_{L^2}\big(\|\nabla U\|_{L^2}+\|\nabla U^e\|_{L^2}\big)ds\\
				&\leq C\nu^{\f12-\frac{1}{r}}\int_0^t \|U-U^e\|_{L^2}\big(\|\omega\|_{L^2}+\|\omega^e\|_{L^2}\big)ds\\
				&\leq C\int_0^t \|U-U^e\|_{L^2}^2ds +C\nu^{1-\frac{2}{r}}\int_0^t (\|\omega\|_{L^2}^2+\|\omega^e\|_{L^2}^2 )ds\\
				&\leq C\int_0^t \|U-U^e\|_{L^2}^2ds +C\nu^{2\alpha-\frac{2}{r}}.
			\end{align*}
			
Next, we deal with $J_{53}$. The fact $\omega^e\in L^\infty_c$ implies that there exists $C_\star$ independent of $2\leq p<\infty$ such that
			\begin{align*}
				\|\nabla U^e\|_{L^p}
				\leq C_\star p\|\omega^e\|_{L^p}.
			\end{align*}
			Thus, for $\beta>0$ small enough, it holds that 
			\begin{align*}
				\int_{B_2} e^{\beta|\nabla U^e|}dxdy			&\leq \sum_{k\geq0}\int_{[-3A,3A]\times [0,50]} \f{\beta^k|\nabla U^e|^k}{k!} dxdy
				\leq C+
				C\sum_{k\geq2}\frac{(k\beta)^k\|\omega^e\|_{L^k}^k}{k!}  \\
				&
				\leq C+
				C\sum_{k\geq2}\frac{(k\beta)^k\|\omega_0\|_{L^\infty}^k}{k!}
				\leq  C+C\sum_{k\geq2}(e\beta\|\omega_0\|_{L^\infty})^k k^{-1/2} ,
			\end{align*} 
			which is convergent provided that $e\beta\|\omega_0\|_{L^\infty}$ is small enough. Then we invoke the following inequality  
			\begin{align*}
				ab\leq e^a+b\log b,
			\end{align*}
			with $a=\beta|\nabla U^e|+\log(\nu^{1/4})$ and $b=|U-U^e|^2$ to obtain
			\begin{align*}
				J_{53}
				&=\frac{1}{\beta}\int_0^t\int_{B_2}\big(\beta|\nabla U^e|
				+\log(\nu^{1/4})
				+\log(\nu^{-1/4}) \big)|U-U^e|^2dxdy\\
				&\leq \frac{1}{\beta}
				\int_0^t\int_{B_2} \nu^{1/4}e^{\beta|\nabla U^e|}
				+|U-U^e|^2 \log|U-U^e| dxdy
				+\frac{1}{\beta}\log(\nu^{-1/4}) 
				\int_0^t \|U-U^e\|^2_{L^2}ds\\
				&\leq C\nu^{1/4}+C\log(\nu^{-1})
				\int_0^t \|U-U^e\|^2_{L^2}ds.
			\end{align*}

			Collecting these estimates together shows that for $r$ large enough, there exists $C$ such that
			\begin{align*}
				\|U(t)-U^e(t)\|_{L^2}^2
				\leq C\nu^{\al}
				+C\log(\nu^{-1})\int_0^t \|U-U^e\|_{L^2}^2ds,
			\end{align*}
			which implies that for some $C'>0$
			\begin{align*}
\|U(t)-U^e(t)\|_{L^2}\leq C\nu^{\f{\al}{2}-C' t} .
			\end{align*}

\end{proof}

\section{Uniform boundedness of the velocity}

\subsection{The initial layer}\label{Sec: initial layer}

When handling the vorticity near the boundary, since the initial data are prescribed in terms of vorticity, the initial velocity may not satisfy the no-slip boundary condition. Consequently, an initial layer arises, necessitating the construction of the following initial layer corrector. The tangential initial velocity on the boundary is directly given by \eqref{BS law integral} as
\begin{align}\label{initial velocity on boundary}
	u_0(x)=\frac{1}{\pi}\int_{\mathbb R^2_+}\frac{y_2}{(x-y_1)^2+y_2^2}\omega_0(y_1,y_2)dy_1dy_2.
\end{align}
We define the initial layer corrector $u_c$ as
\begin{align}\label{initial layer corrector velocity}
	\left\{
	\begin{aligned}
		&\pa_t u_c-\nu\pa_y^2 u_c=0,\\
		&u_c|_{t=0}=u_0(x)\chi(y),\qquad u_c|_{y=0}=0,
	\end{aligned}
	\right.
\end{align}
where the cut-off function $\chi$ is defined in \eqref{def:chi}.
Thus, $u_c$ obeys the following expression 
\begin{align*}
	u_c(t,x,y)=u_0(x)\int_0^{+\infty}\frac{1}{(4\pi \nu t)^{1/2}}\big(e^{-\frac{(y-z)^2}{4\nu t}}-e^{-\frac{(y+z)^2}{4\nu t}}\big)\chi(z)dz.
\end{align*}
Furthermore, we define the vorticity corrector as
\begin{align}\label{initial layer corrector vorticity}
	\omega_c(t,x,y)
	&=-\pa_y u_c(t,x,y)\nonumber\\
	&=-\frac{2u_0(x)}{(4\pi \nu t)^{1/2}}e^{-\frac{y^2}{4\nu t}}
	-\int_0^{+\infty}
	\frac{u_0(x)}{(4\pi\nu  t)^{1/2}}\big(e^{-\frac{(y-z)^2}{4\nu t}}+e^{-\frac{(y+z)^2}{4\nu t}}\big)\chi'(z)dz.
\end{align}

Through direct calculation (for further details, see Lemma 8.3 in \cite{WYZ}), we derive the following estimates.
\begin{lemma}\label{est of omega c}
	For $\eps_0, T_0, \delta$ small enough, there exists $C_0, C'$ 
	\begin{align}\label{boundedness of omega c}
	\sup_{0<t<T_0}\left\|\left\|e^{C'|\xi|}e^{\frac{C'y^2}{\nu t}}\big((1,x)\pa_x^i(y\pa y)^j\pa_y^k\omega_c(t)\big)_\xi\right\|_{L^1_y}\right\|_{L^1_\xi\cap L^2_\xi}
	\leq C_0 (\nu t)^{-k/2},\quad \text{for}\ i,j,k\geq0,
\end{align}
\begin{align}\label{boundedness of omega c 2}
	\left\|e^{C'|\xi|}(\pa_y+|\xi|)(\omega_c)_\xi|_{y=0}(t)\right\|_{L^1_\xi\cap L^2_\xi}
	\leq C_0 (\nu t)^{-1/2},
\end{align}
where $(\omega_c)_\xi$ stands for the Fourier transform in $x$ of $\omega_c$.
\end{lemma}

\begin{remark}
	Owing to the initial layer, it has been proven in \cite{Ken} that the solution $\omega$  of the Navier-Stokes system can be expressed as
	\begin{align}\label{solu in Ken}
		\omega(t)=\omega_{cont}-\frac{2}{(4\pi \nu t)^{1/2}}e^{-\frac{y^2}{4\nu t}}e^{t\pa_x^2}u_0,
	\end{align}
	where $\omega_{cont}$ denotes a function  continuous at $t=0$, and $u_0$ is defined by \eqref{initial velocity on boundary}. For the second part in \eqref{solu in Ken} and the first part in \eqref{initial layer corrector vorticity}, letting $t\rightarrow 0^+$, we obtain
	\begin{align*}
		\lim_{t\rightarrow 0^+}-\frac{2}{(4\pi \nu t)^{1/2}}e^{-\frac{y^2}{4\nu t}}e^{t\pa_x^2}u_0
		=\lim_{t\rightarrow 0^+} -\frac{2u_0(x)}{(4\pi \nu t)^{1/2}}e^{-\frac{y^2}{4\nu t}}
		=-u_0(x)\delta_{\partial\mathbb R^2_+}.
	\end{align*}
	Consequently,  $\omega-\omega_c$ is continuous at $t=0$.
\end{remark}

 \subsection{Functional spaces}\label{Section Functional framework}
 
To control the vorticity near the boundary, we introduce the following norms
\begin{align*}
	\| f\|_{\mu,t}
	=\int_0^{1+\mu}e^{\eps_0(1+\mu)\frac{y^2}{\nu(1+t)}}|f(y)|dy,\quad
	\|f\|_{Y^k_{\mu,t}}
	=\left\|\|e^{\eps_0(1+\mu-y)_+|\xi|}f_\xi\|_{\mu,t}\right\|_{L^k_\xi}, \ k=1,2.
\end{align*}
Now we define 
\begin{align}\label{def: Y}
	\|f\|_{Y_k(t)}=\sup_{\mu<\mu_0-\gamma t}\Big(
	 \sum_{i+j\leq1}\|\pa_x^i(y\pa_y)^jf\|_{Y^k_{\mu,t}} +(\mu_0-\mu-\gamma t)^\alpha\sum_{i+j=2}\|\pa_x^i(y\pa_y)^jf\|_{Y^k_{\mu,t}} \Big).
\end{align}
Here $\mu\leq\mu_0=\frac{1}{10}$, $\eps_0\ll1$ and $\gamma\gg1$ are constants to be determined later in the proof, $\alpha\in(\frac{1}{2},1)$ is a fixed constant. Throughout this paper, we suppose $t\in(0,\frac{1}{2\gamma})$, and $f_\xi$ denotes the Fourier transform in $x$ variable of  $f(x,y)$.

We also need to control the vorticity in the region away from the boundary. Let $\chi_0(y)$ be a smooth cut-off function satisfying 
\begin{align}\label{def of chi0}
	\chi_0(y)=
	\left\{
	\begin{aligned}
		&1,\quad y\geq\f38,\\
		&0,\quad y\leq\f14.
	\end{aligned}
	\right.
\end{align}
Let $\theta(y)$ be a smooth function such that $\theta(y)$ is decreasing on $[0,3]$, increasing on $[3,+\infty)$ and satisfies
\begin{align}\label{def of theta}
	\theta(y)=1 \quad\text{for}\quad y\leq\frac{3}{8}\quad\text{or}\quad y\geq6,\quad
	\theta(3)=0,\quad
	\theta(y)\leq\frac{1}{4} \quad\text{for}\quad \frac{1}{2}\leq y\leq 5.
\end{align}
Choose $T_0$ small enough such that for $0\leq t\leq T_0$, there exist
\begin{align*}
	y_1(t)\in\big(\frac{8}{32},\frac{14}{32}\big)\quad\text{and}\quad
	y_2(t)\in\big(\frac{11}{2},6\big)
\end{align*}
 with
$1-\gamma t-\theta(y_1(t))=1-\gamma t-\theta(y_2(t))=0$. We next  introduce the weights
\begin{align}\label{def of weight}
	\Psi(t,y)=\frac{20\eps_0}{\nu}\big(1-\gamma t-\theta(y)\big)_+,\quad
	\psi(y)=y^2.
\end{align}
In this paper, we use norms $\|e^\Psi\chi_0\psi\omega\|_{L^2}$ and $\|\chi_0\omega\|_{L^p}$ to control the vorticity in the region away from the boundary. 

We then introduce the following energy functional 
\begin{align}\label{energy 1}
	E(t):=\sup_{0\leq s\leq t}e(s),
	\end{align}
	with	
\begin{align*}	
	\quad e(t):=\|(1,x)(\omega(t)-\omega_c(t))\|_{Y_1(t) \cap Y_2(t)}
	+\|e^\Psi\chi_0\psi\omega(t)\|_{L^2}
	+\|\chi_0\omega(t)\|_{L^p},
	\end{align*}
 where $2<p<+\infty$.

\subsection{Energy estimates}

The proof of Proposition \ref{prop: uniform for U omega NS} relies on the following propositions and lemma.

\begin{proposition}\label{prop: Y(t)}
For $0<t<T$ small enough, we have
	\begin{align}\label{est: Y(t)}
&\|(1,x)(\omega(t)-\omega_c(t))\|_{Y_1(t) \cap Y_2(t)}\nonumber\\
&\leq  \f{C}{\gamma^{\f12}} \big(E(t)+1\big)^2+Ce^{\frac{4\eps_0}{\nu}}\sup_{[0,t]}\| (1,x)\omega(s)\|_{H^3(\frac{7}{8}\leq y\leq4)}^2+C\nu^{1/2},
\end{align}
and
\begin{align}\label{est: omega near boundary L infty}
				\Big\|\sup_{0<y<\frac{3}{4}}|\omega_\xi(t,y)|\Big\|_{L^2_\xi}
				\leq C(\nu t)^{-1/2}
				+C\nu^{-\f12}\Big(\big(E(t)+1\big)^2+e^{\frac{4\eps_0}{\nu}}\sup_{[0,t]}\|(1,x)\omega(s)\|^2_{H^3(\frac{7}{8}\leq y\leq4)} \Big) .
			\end{align}
\end{proposition}

\begin{proposition}\label{weighted est 1}
	There exists $T_0$ small enough such that for $0\leq t\leq T_0$,
	\begin{align*}
		&\sup_{[0,t]}\|e^\Psi\chi_0\psi\omega\|^2_{L^2}
		+\gamma\int_0^t\frac{20\eps_0}{\nu}\left\|e^\Psi\chi_0\psi\omega\right\|^2_{L^2(y_1(s)\leq y\leq y_2(s))}
		ds\\
		&\leq \frac{CE(t)}{\nu}\int_0^t \left\|e^\Psi\chi_0\psi\omega\right\|^2_{L^2(y_1(s)\leq y\leq y_2(s))}ds
		+Ct\big(E(t)+1\big)^3
		+\frac{C\nu}{\gamma}E(t)^2
		+C\|e^\Psi\chi_0\psi\omega_0\|^2_{L^2}.
	\end{align*}
\end{proposition}

\begin{proposition}\label{L4 est of omega 1}
	There exists $T_0$ small enough such that for $0\leq t\leq T_0$, $2<p<+\infty$,
	\begin{align*}
		\sup_{[0,t]}\|\chi_0\omega\|_{L^p}
		\leq C\|\chi_0\omega_0\|_{L^p}
		+C(t+\frac{\nu}{\gamma})^{\f1p}\big(E(t)+1\big)^{\f{p+1}{p}}.
	\end{align*}
\end{proposition}

\begin{proposition}\label{Sobolev estimate}
	There exists $T_0$ small enough such that for $0\leq t\leq T_0$,
	\begin{align*}
		\sup_{[0,t]}\| (1,x)\omega(s)\|_{H^3(\frac{7}{8}\leq y\leq 4)}
		\leq C\nu t^{\frac{1}{2}}\big(E(t)+1\big)^8 e^{-\frac{5\eps_0}{\nu}}.
	\end{align*}
\end{proposition}

\medskip

We also need the following velocity estimates to handle the transport terms.

\begin{lemma}\label{velocity estimates 1}
	It holds that
	
	(1) For $i=0,1$
	 \begin{align*}
		\left\|\sup_{0<y<1+\mu}e^{\eps_0(1+\mu-y)_+|\xi|}|\pa_x^i u_\xi(s)|\right\|_{L^1_\xi}
		\leq C\big(e(s)+\|\omega(s)\|_{H^{1+i}(1\leq y\leq2)}+1\big),
	\end{align*}
	
	(2)
	\begin{align*}
		\left\|\sup_{0<y<1+\mu}e^{\eps_0(1+\mu-y)_+|\xi|}\left|\frac{v_\xi(s)}{y}\right|\right\|_{L^1_\xi}
		\leq C\big(e(s)+\|\omega(s)\|_{H^1(1\leq y\leq2)}+1\big),
	\end{align*}
	and
	\begin{align*}
		\left\|\sup_{0<y<1+\mu}e^{\eps_0(1+\mu-y)_+|\xi|}\left|\frac{(\pa_x v)_\xi(s)}{y}\right|\right\|_{L^1_\xi}
		\leq C\big((\mu_0-\mu-\gamma s)^{-\alpha}e(s)
		+\|\omega(s)\|_{H^2(1\leq y\leq2)}+1\big).
	\end{align*}
	
	(3)
	\begin{align*}
		\left\|\sup_{0<y<1+\mu}e^{\eps_0(1+\mu-y)_+|\xi|}\left|y\pa_y \big(u_\xi(s),\frac{v_\xi(s)}{y} \big)\right|\right\|_{L^1_\xi}
		\leq C\big(e(s)+\|\omega(s)\|_{H^2(1\leq y\leq2)}+1\big).
	\end{align*}
	
	(4) For $i+j\leq2$,
	\begin{align*}
		\left\|\pa_x^i\pa_y^j U(s)\right\|_{L^\infty(1\leq y\leq3)}
		\leq C\big(e(s)+\|\omega(s)\|_{H^{i+j+1}(\frac{7}{8}\leq y\leq4)}+1\big).
	\end{align*}
	
	(5)\begin{align*}
		\|U(s)\|_{L^\infty}\leq C\big(e(s)+1\big).
	\end{align*}
\end{lemma}

Propositions \ref{prop: Y(t)}–\ref{Sobolev estimate} and Lemma \ref{velocity estimates 1} will be proven in subsequent sections.

\subsection{Proof of Proposition \ref{prop: uniform for U omega NS}}  By Proposition \ref{prop: Y(t)}-Proposition \ref{Sobolev estimate}, we obtain
\begin{align*}
	&E(t)+\gamma^{\f12}\left\{\int_0^t\frac{20\eps_0}{\nu}\left\|e^\Psi\chi_0\psi\omega\right\|^2_{L^2(y_1(s)\leq y\leq y_2(s))}ds
		 \right\}^{\f12}\\
		&\leq CE(t)^{\f12}\left\{\int_0^t\frac{1}{\nu}\left\|e^\Psi\chi_0\psi\omega\right\|^2_{L^2(y_1(s)\leq y\leq y_2(s))}ds
		 \right\}^{\f12} 
		+C\|\chi_0\omega_0\|_{L^p}
		+\frac{C}{\gamma^{1/p}}\big(E(t)+1\big)^{16}+C\nu^{1/2}.
\end{align*}
Then the continuous argument ensures that there exist constants $T_0$ and $\gamma$ such that
\ben\label{est:E}
E(T_0) \leq C.
\een
Then we get by (5) in Lemma \ref{velocity estimates 1} that 
	\ben\label{est:U}
	 \sup_{[0,T_0]}\|U(t) \|_{L^\infty}\leq C,
\een
which implies the first part of Proposition \ref{prop: uniform for U omega NS}.  The second part follows from the following lemma. 

\begin{lemma}\label{lem:ov}
			Under the assumptions of Theorem \ref{thm: Lp convergence}, for $\nu$ small enough, there holds
			\begin{align*}
				\int_0^{T_0}\|\omega\|^2_{L^2}ds
				\leq C  \nu^{-\f12}.
			\end{align*}
		\end{lemma}
		
		\begin{proof}
			A direct computation, combined with Lemma \ref{est of omega c}, yields
			\begin{align*}
				\int_0^t\|\omega\|^2_{L^2}ds
				&\leq \int_0^t\|\omega\|^2_{L^2(y\leq\frac{3}{4})}ds
				+\int_0^t\|\omega\|^2_{L^2(y\geq\frac{3}{4})}ds\\
				&\leq \int_0^t \int_{\mathbb R}\int_0^{3/4}|\omega_\xi(s,y)|^2 dyd\xi ds
				+\int_0^t\|e^\Psi\chi_0\omega\|^2_{L^2}ds\\
				&\leq \int_0^t \left\|\sup_{0<y<\frac{3}{4}}|\omega_\xi(s,y)|\right\|_{L^2_\xi}\left\|\int_0^{3/4}|(\omega-\omega_c+\omega_c)_\xi(s,y)|dy\right\|_{L^2_\xi}ds
				+CtE(t)\\
				&\leq C\big(E(t)+1\big)\int_0^t\left\|\sup_{0<y<\frac{3}{4}}|\omega_\xi(s,y)|\right\|_{L^2_\xi}ds
				+CtE(t),
			\end{align*}
			which along with \eqref{est: omega near boundary L infty}, \eqref{est:E} and Proposition \ref{Sobolev estimate} implies 
			\begin{align*}
				\int_0^t\|\omega\|^2_{L^2}ds
				&\leq C\big(E(t)+1\big)\int_0^t \{ (\nu s)^{-\f12}
				+\nu^{-\f12}\Big(\big(E(s)+1\big)^2+e^{\frac{4\eps_0}{\nu}}\sup_{[0,s]}\|(1,x)\omega(\tau)\|^2_{H^3(\frac{7}{8}\leq y\leq4)} \Big) \} ds\\
				&\quad+CtE(t)
				\leq C\nu^{-1/2},
			\end{align*}
			where one takes $\nu$ small enough.
		\end{proof}

\section{Estimates of the vorticity near the boundary}\label{sec: Estimates near the boundary}

This section is devoted to the proof of Proposition \ref{prop: Y(t)}. We first introduce the vorticity $\omega=\pa_x v-\pa_y u$, which satisfies 
\begin{align}\label{eq: NS vorticity}
\pa_t\omega+U\cdot\nabla\omega=&\nu\Delta\omega,
\end{align}
where $U=\nabla^\perp\Delta_D^{-1}\omega$. The boundary condition introduced in \cite{Maekawa}  is given by
\begin{align}\label{eq: BC of NS}
	\nu(\pa_y+|D_x|)\omega|_{y=0}=\pa_y\Delta_D^{-1}(U\cdot\nabla\omega)|_{y=0}.
\end{align}

\subsection{Representation formula}

Since the functional space $Y_k(t)$ reveals the behavior of the vorticity near the boundary, we derive the system of $\chi (\om-\omega_c)$ by multiplying $\chi$ on both sides of \eqref{eq: NS vorticity}
 to arrive at
\begin{align}\label{eq: eq of omega near the boundary}
	\left\{
	\begin{aligned}
		&\pa_t(\chi\omega-\chi\omega_c)-\nu\Delta(\chi\omega-\chi\omega_c)
	=N,\\
	&\chi\omega-\chi\omega_c|_{t=0}=u_0\chi\chi':=b ,\\
	&\nu(\pa_y+|D_x|)(\chi\omega-\chi\omega_c)|_{y=0}=\pa_y\Delta_D^{-1}(U\cdot\nabla\omega)|_{y=0}-\nu(\pa_y+|D_x|)\omega_c|_{y=0}:=B,
	\end{aligned}
	\right.
\end{align}
where $N$ is defined by
\begin{align}\label{def of N}
	N=-\chi U\cdot\nabla\omega
	+\nu\chi\pa_x^2\omega_c
	-(\nu\chi''\omega+2\nu\chi'\pa_y\omega)
	+(\nu\chi''\omega_c+2\nu\chi'\pa_y\omega_c).
\end{align}

For $\chi x(\omega-\omega_c)$, we have
\begin{align}\label{eq: eq of x omega near the boundary}
	\left\{
	\begin{aligned}
		&\pa_t(\chi x\omega-\chi x\omega_c)-\nu\Delta(\chi x\omega-\chi x\omega_c)
	=\widetilde N,\\
	&\chi x\omega-\chi x\omega_c|_{t=0}=x u_0\chi\chi':=\tilde b ,\\
	&\nu(\pa_y+|D_x|)(\chi x\omega-\chi x\omega_c)|_{y=0}=\widetilde B,
	\end{aligned}
	\right.
\end{align}
where $\widetilde N$ is defined by
\begin{align}\label{def of tilde N}
	\widetilde N=&-\chi x U\cdot\nabla\omega+\nu\chi x\pa_x^2\omega_c
	-2\nu\chi\pa_x(\omega-\omega_c)\\
	\nonumber
	&-(\nu\chi''x\omega+2\nu\chi'x\pa_y\omega)
	+(\nu\chi''x\omega_c+2\nu\chi'x\pa_y\omega_c).
\end{align}
A direct computation gives
\begin{align}\label{def of tilde B}
	\widetilde B_\xi&=\nu(\pa_y+|\xi|)(x\omega-x\omega_c)_\xi|_{y=0} 
	=i\nu(\pa_y+|\xi|)\pa_\xi(\omega-\omega_c)_\xi|_{y=0} \\
	\nonumber
	&=i\nu\pa_\xi\big((\pa_y+|\xi|)(\omega-\omega_c)_\xi\big)|_{y=0}-i\nu sgn\xi (\omega-\omega_c)_\xi|_{y=0} \\
	\nonumber
	&=i\pa_\xi(B_\xi)-i\nu sgn\xi(\omega-\omega_c)_\xi|_{y=0}.
\end{align}

 By the solution formula derived in \cite{Maekawa}, we get 
 \begin{align}\label{eq: solution formula of omega near the boundary}
	(\chi\omega-\chi\omega_c)_\xi(t,y)
	&=\int_0^{+\infty}\big(H_\xi(t,y,z)+R_\xi(t,y,z)\big)b_\xi(z)dz \\
	\nonumber
	&+\int_0^t\int_0^{+\infty}H_\xi(t-s,y,z)N_\xi(s,z)dzds
	-\int_0^t H_\xi(t-s,y,0)B_\xi(s)ds\\
	\nonumber
	&+\int_0^t\int_0^{+\infty}R_\xi(t-s,y,z)N_\xi(s,z)dzds
	-\int_0^t R_\xi(t-s,y,0)B_\xi(s)ds,
\end{align}
\begin{align}\label{eq: solution formula of x omega near the boundary}
	(\chi x\omega-\chi x\omega_c)_\xi(t,y)
	&=\int_0^{+\infty}\big(H_\xi(t,y,z)+R_\xi(t,y,z)\big)\tilde b_\xi(z)dz \\
	\nonumber
	&+\int_0^t\int_0^{+\infty}H_\xi(t-s,y,z)\widetilde N_\xi(s,z)dzds
	-\int_0^t H_\xi(t-s,y,0)\widetilde B_\xi(s)ds\\
	\nonumber
	&+\int_0^t\int_0^{+\infty}R_\xi(t-s,y,z)\widetilde N_\xi(s,z)dzds
	-\int_0^t R_\xi(t-s,y,0)\widetilde B_\xi(s)ds,
\end{align}
where
 \begin{align}
 	&H_\xi(t,y,z)=e^{-\nu\xi^2t}\big(g(\nu t,y-z)+g(\nu t,y+z)\big), \label{def of H xi}\\
	 &R_\xi(t,y,z)=\big( \Gamma (\nu t,x,y+z)-\Gamma(0,x,y+z)\big)_\xi, \label{def of Gamma xi}
 \end{align}
 with 
 \beno
g(t,x)=\frac{1}{(4\pi t)^{1/2}}e^{-\frac{x^2}{4t}},\quad  \Gamma(t,x,y)=\big(\Xi E\ast G(t)\big)(x,y).
   \eeno
 Here
 \beno
 \Xi=2(\pa_x^2+|D_x|\pa_y\big)
   ,\
   E(x)=-\frac{1}{2\pi}\operatorname{log}|x|,\quad G(t,x,y)=g(t,x)g(t,y).
 \eeno

In \cite{KVW}, \cite{Maekawa} and \cite{NN}, $R_\xi$ enjoys the following properties.
 \begin{lemma}\label{prop of R 1}
 	(1) $\pa_y R_\xi(t,y,z)=\pa_z R_\xi(t,y,z).$
 	
 	(2)
 	\begin{align*}
 		&|\pa_z^k R_\xi(t,y,z)|\leq C a^{k+1}e^{-\theta_0 a(y+z)}
 		+\frac{C}{(\nu t)^{(k+1)/2}}e^{-\theta_0\frac{(y+z)^2}{\nu t}}e^{-\frac{\nu\xi^2t}{8}},\quad k\geq0,\quad a=|\xi|+\frac{1}{\sqrt{\nu}}.\\
 		&|(y\pa_y)^kR_\xi(t,y,z)|\leq Cae^{-\frac{\theta_0}{2}a(y+z)}
 		+\frac{C}{\sqrt{\nu t}}e^{-\frac{\theta_0}{2}\frac{(y+z)^2}{\nu t}}e^{-\frac{\nu\xi^2t}{8}},\quad k=0,1,2,
 	\end{align*}
 	where $\theta_0$ is a universal constant and $C$ depends only on $\theta_0$.
	
 	(3)
		\begin{align*}
 		&\int_0^t\int_0^{+\infty}R_\xi(t-s,y,z)\big(f_\xi(s,z)-h_\xi(s,z)\delta_{z=0}\big)dzds\\
 		&=2\nu\int_0^t\int_0^s\int_0^{+\infty}(\xi^2-|\xi|\pa_y)\Big(e^{-\nu(s-\tau)\xi^2}g\big(\nu(s-\tau),y+z\big)\Big)
 		\big(f_\xi(\tau,z)-h_\xi(\tau,z)\delta_{z=0}\big)dzd\tau ds.
 	\end{align*}
 \end{lemma}
 
 \begin{remark}
 By Lemma \ref{prop of R 1}, we have
 	\begin{align*}
 		R_\xi(t,y,z)=2\nu\int_0^t(\xi^2-|\xi|\pa_y)\Big(e^{-\nu s\xi^2}g(\nu s,y+z)\Big)ds.
 	\end{align*}
 \end{remark}
 
The following lemma provides estimates of $b$ and $\tilde b$ in \eqref{eq: eq of omega near the boundary} and \eqref{eq: eq of x omega near the boundary}.
\begin{lemma}\label{lem: est of initial b}
	There exists $C',T>0$ such that for $t\in[0,T],j\leq10$, it holds
	\begin{align*}
		\left\|\left\|e^{C'|\xi|}e^{\frac{C'y^2}{\nu t}}
		\int_0^{+\infty}(y\pa_y)^j\big(H_\xi(t,y,z)+R_\xi(t,y,z)\big)\big(b_\xi(z),\tilde b_\xi(z)\big) dz\right\|_{L^1_y\cap L^\infty_y(y\leq3/2)}\right\|_{L^1_\xi\cap L^2_\xi}
		\leq Ce^{-\frac{C'}{\nu t}}.
	\end{align*}
\end{lemma}

\begin{proof}
	Thanks to the definition of $b$, we have
	\begin{align*}
		\big(b_\xi(y),\tilde b_\xi(y)\big)
		 = \frac{\chi(y)\chi'(y)}{\pi}\int_{ \mathbb R^2_+}(1,-2\pi y_2 sgn\xi)e^{2\pi iy_1\xi-2\pi y_2|\xi|}\omega_0(y_1,y_2)dy_1dy_2,
	\end{align*}
	which along with $H_\xi, R_\xi$ in \eqref{def of H xi}, \eqref{def of Gamma xi} and Lemma \ref{prop of R 1} implies the desired result.
\end{proof}

\subsection{Some basic estimates}

For the last four parts of $(\chi\omega-\chi\omega_c)_\xi(t,y)$, we have the following lemmas to estimate them in $Y_k(t)$ space. To simplify the notations, we introduce 
\begin{align}\label{def:N-W}
	\|N(t)\|_{W_{\mu, t}}:=\sum_{i+j\leq1}\|\pa_x^i(y\pa_y)^j N(t)\|_{Y^1_{\mu,t}\cap Y^2_{\mu,t}}+e^{\frac{2\eps_0}{\nu}}\sum_{i+j\leq2}\|\|\pa_x^i\pa_y^j N(t)\|_{L^2_x} \|_{L^1_y(y\geq1+\mu)}.
\end{align}

\begin{lemma}\label{lem: om-1}
	For $\mu<\mu_0-\gamma t$ and $\mu_1=\mu+\frac{1}{2}(\mu_0-\mu-\gamma s)$, we have 
	\beno
	&&\left\|\int_0^t\int_0^{+\infty}H(t-s,y,z)N(s,z)dzds \right\|_{Y_1(t)\cap Y_2(t)} \\
	&&\leq C\sup_{\mu<\mu_0-\gamma t}\int_0^t\|N(s)\|_{W_{\mu, s}}ds \\
	&&\quad+ C\sup_{\mu<\mu_0-\gamma t}(\mu_0-\mu-\gamma t)^\al\int_0^t\Big((\mu_0-\mu-\gamma s)^{-1}+(\mu_0-\mu-\gamma s)^{-\frac{1}{2}}(t-s)^{-\frac{1}{2}}\Big)\|N(s)\|_{W_{\mu_1, s}}ds.
	\eeno

\end{lemma}	
	 
\begin{lemma}\label{lem: om-2}
	For $\mu<\mu_0-\gamma t$ and $\mu_1=\mu+\frac{1}{2}(\mu_0-\mu-\gamma s)$, we have
	\beno
	&&\left\|\int_0^t H(t-s,y,0)B(s)ds\right\|_{Y_1(t)\cap Y_2(t)} \\
	&&\leq C\sup_{\mu<\mu_0-\gamma t}\int_0^t\sum_{i\leq1}\left\| e^{\eps_0(1+\mu)|\xi|}\xi^i B_\xi(s)\right\|_{L^1_\xi\cap L^2_\xi}ds\\
	&&\quad+C\sup_{\mu<\mu_0-\gamma t} (\mu_0-\mu-\gamma t)^\al\int_0^t(\mu_0-\mu-\gamma s)^{-1}
		\sum_{i\leq1}\left\| e^{\eps_0(1+\mu_1)|\xi|}\xi^i B_\xi(s)\right\|_{L^1_\xi\cap L^2_\xi}ds.
	\eeno
\end{lemma}	

 \begin{lemma}\label{lem: om-3}
	For $\mu<\mu_0-\gamma t$ and $\mu_1=\mu+\frac{1}{2}(\mu_0-\mu-\gamma s)$, we have
	\beno
	&&\left\|\int_0^t\int_0^{+\infty}R(t-s,y,z)N(s,z)dzds\right\|_{Y_1(t)\cap Y_2(t)}\\
	&& \leq \frac{C}{t}\sup_{\mu<\mu_0-\gamma t} (\mu_0-\mu-\gamma t)^\al\int_0^t(\mu_0-\mu-\gamma s)^{-1}
		\int_0^s \|N(\tau)\|_{W_{\mu_1, \tau}}d\tau ds\\
&&+	C\sup_{\mu<\mu_0-\gamma t}\int_0^t\|N(s)\|_{W_{\mu, s}}ds  +C\sup_{\mu<\mu_0-\gamma t} (\mu_0-\mu-\gamma t)^\al\int_0^t(\mu_0-\mu-\gamma s)^{-1} \|N(s)\|_{W_{\mu_1, s}} ds
	\eeno
\end{lemma}	

\begin{lemma}\label{lem: om-4}
	For $\mu<\mu_0-\gamma t$, we have
	\beno
	&&\left\|\int_0^t R(t-s,y,0)B(s)ds\right\|_{Y_1(t)\cap Y_2(t)} \\
	&&		\leq \frac{C}{t}\sup_{\mu<\mu_0-\gamma t}(\mu_0-\mu-\gamma t)^\al\int_0^t(\mu_0-\mu-\gamma s)^{-1}\int_0^s
		\sum_{i\leq1}\left\|e^{\eps_0(1+\mu_1)|\xi|}\xi^iB_\xi(\tau)\right\|_{L^1_\xi\cap L^2_\xi}d\tau ds\\
&&\quad+C\sup_{\mu<\mu_0-\gamma t}\int_0^t
		\sum_{i\leq1}\left\|e^{\eps_0(1+\mu)|\xi|}\xi^iB_\xi(s)\right\|_{L^1_\xi\cap L^2_\xi}ds\\
		&&\quad+C\sup_{\mu<\mu_0-\gamma t} (\mu_0-\mu-\gamma t)^\al\int_0^t(\mu_0-\mu-\gamma s)^{-1}
		\sum_{i\leq1}\left\|e^{\eps_0(1+\mu_1)|\xi|}\xi^iB_\xi(s)\right\|_{L^1_\xi\cap L^2_\xi}ds.
		\eeno
\end{lemma}	

The proof of Lemma \ref{lem: om-1}--Lemma \ref{lem: om-4} is postponed to the end of this section. To obtain the estimates of $\|(1,x)\big(\om-\om_c\big)\|_{Y_1(t)\cap Y_2(t)}$, it remains to give the estimates of $N$ and $B$.

\subsection{Estimates of $N$ and $B$}
Recall that $N, \widetilde N$ are defined in \eqref{def of N} and \eqref{def of tilde N} respectively.
\begin{lemma}\label{est of N in new norm}
	For $0<\mu<\mu_0-\gamma s$, it holds that
	\begin{align*}
		&\left\|\big(N(s),\widetilde N(s)\big)\right\|_{W_{\mu,s}}\\
		&\leq C(\mu_0-\mu-\gamma s)^{-\alpha}\big(e(s)+1\big)^2
		+C\big((\mu_0-\mu-\gamma s)^{-\alpha}+e^{\frac{4\eps_0}{\nu}}\big)\| (1,x)\omega(s)\|_{H^3(\frac{7}{8}\leq y\leq4)}^2.
	\end{align*}
\end{lemma}

The proof of Lemma \ref{est of N in new norm} follows from Lemma \ref{est of N1} and Lemma \ref{est of omega c}, and by taking $\eps_0$ sufficiently small.

\begin{lemma}\label{est of N1}
	For $0<\mu<\mu_0-\gamma s$, it holds that
	\begin{align*}
 \sum_{i+j\leq1}\left\|\pa_x^i(y\pa_y)^j \big(N(s),\widetilde N(s)\big) \right\|_{Y^1_{\mu,s}\cap Y^2_{\mu,s}}  
 \leq C(\mu_0-\mu-\gamma s)^{-\alpha}\Big(\big(e(s)+1\big)^2
		+\| (1,x)\omega(s)\|_{H^2(1\leq y\leq2)}^2 \Big),
	\end{align*}
	and
		\begin{align*}
 \sum_{i+j\leq2}\left\|\left\|\pa_x^i\pa_y^j   \big(N(s),\widetilde N(s)\big) \right\|_{L^2_x}\right\|_{L^1_y(y\geq1)}  &\leq C\Big(\big(e(s)+1\big)^2
		+\| (1,x)\omega(s)\|_{H^2(1\leq y\leq2)}^2 \Big)
		+C\nu\|\omega_c\|_{H^4(1\leq y\leq3)}.
	\end{align*}
	For the first inequality and the case $i=j=0$, the factor $(\mu_0-\mu-\gamma s)^{-\alpha}$can be removed.
\end{lemma}

\begin{proof}
We only prove for $N$, since $\widetilde N$ can be proved in a same way.
First of all, we deal with the first inequality. By the definition of $Y^k_{\mu, s}$, we only deal with  the strip $0\leq y\leq 1+\mu$.

 Due to the definition of $\chi$, we note that $N=-U\cdot\nabla\omega+\nu\pa_x^2\omega_c$ for $0<y<1+\mu$.\smallskip

	\underline{Case 1: $i=j=0.$}
	Lemma \ref{product estimate} gives
	\begin{align*}
		&\|N(s)\|_{Y^k_{\mu, s}}
		\leq 
		\|u\pa_x\omega\|_{Y^k_{\mu,s}}
		+\|v\pa_y\omega\|_{Y^k_{\mu,s}}
		+\nu\|\pa_x^2\omega_c\|_{Y^k_{\mu,s}} \\
		\nonumber
		&\leq \left\|\sup_{0<y<1+\mu}e^{\eps_0(1+\mu-y)_+|\xi|}|u_\xi(s,y)|\right\|_{L^1_\xi}
		\Big(\|\pa_x(\omega-\omega_c)(s)\|_{Y^k_{\mu,s}}
		+\|\pa_x\omega_c(s)\|_{Y^k_{\mu,s}}\Big)+\nu\|\pa_x^2\omega_c\|_{Y^k_{\mu,s}}\\
		\nonumber
		 &\quad+\left\|\sup_{0<y<1+\mu}e^{\eps_0(1+\mu-y)_+|\xi|}\frac{|v_\xi(s,y)|}{y}\right\|_{L^1_\xi}
		\Big(\|y\pa_y(\omega-\omega_c)(s)\|_{Y^k_{\mu,s}}
		+\|y\pa_y\omega_c(s)\|_{Y^k_{\mu,s}}\Big) \\
		\nonumber
		&\leq C\big(e(s)+\|\omega(s)\|_{H^1(1\leq y\leq2)}+1\big)\big(e(s)+1\big),
	\end{align*}
	here we used Lemma \ref{est of omega c} and Proposition \ref{velocity estimates 1} in the last step.\smallskip

	\underline{Case 2: $i+j=1$.}
	Similarly,  we utilize Lemma \ref{product estimate} to obtain
	\begin{align*}
		&\|\pa_x N(s)\|_{Y^k_{\mu,s}}
		\leq \|\pa_x u(s)\pa_x\omega(s)\|_{Y^k_{\mu,s}}
		+\|\pa_x u(s)\pa_x^2\omega(s)\|_{Y^k_{\mu,s}}
		+\left\|\pa_x v(s)\pa_y\omega(s)\right\|_{Y^k_{\mu,s}}\\
		&\qquad\qquad\qquad\qquad+\left\|v(s)\pa_x\pa_y\omega(s)\right\|_{Y^k_{\mu,s}}
		+\nu\|\pa_x^3\omega_c\|_{Y^k_{\mu,s}}\\
		&\leq \left\|\sup_{0<y<1+\mu}e^{\eps_0(1+\mu-y)_+|\xi|}|(\pa_x u)_\xi(s,y)|\right\|_{L^1_\xi}
		\Big(\|\pa_x(\omega-\omega_c)(s)\|_{Y^k_{\mu,s}}
		+\|\pa_x\omega_c(s)\|_{Y^k_{\mu,s}}\Big)\\
		&\quad+\left\|\sup_{0<y<1+\mu}e^{\eps_0(1+\mu-y)_+|\xi|}| u_\xi(s,y)|\right\|_{L^1_\xi}
		\Big(\|\pa_x^2(\omega-\omega_c)(s)\|_{Y^k_{\mu,s}}
		+\|\pa_x^2\omega_c(s)\|_{Y^k_{\mu,s}}\Big)\\
		&\quad+\left\|\sup_{0<y<1+\mu}e^{\eps_0(1+\mu-y)_+|\xi|}\frac{|(\pa_x v)_\xi(s,y)|}{y}\right\|_{L^1_\xi}
		\Big(\|y\pa_y(\omega-\omega_c)(s)\|_{Y^k_{\mu,s}}
		+\|y\pa_y\omega_c(s)\|_{Y^k_{\mu,s}}\Big)\\
		&\quad+\left\|\sup_{0<y<1+\mu}e^{\eps_0(1+\mu-y)_+|\xi|}\frac{|v_\xi(s,y)|}{y}\right\|_{L^1_\xi}
		\Big(\|\pa_x (y\pa_y)(\omega-\omega_c)(s)\|_{Y^k_{\mu,s}}
		+\|\pa_x (y\pa_y)\omega_c(s)\|_{Y^k_{\mu,s}} \Big)
		+\nu\|\pa_x^3\omega_c\|_{Y^k_{\mu,s}}\\
		&\leq C(\mu_0-\mu-\gamma s)^{-\alpha}\big(e(s)+1\big)^2
		+C(\mu_0-\mu-\gamma s)^{-\alpha}\big(e(s)+1\big)\|\omega\|_{H^2(1\leq y\leq2)}+\nu\|\pa_x^3\omega_c\|_{Y^k_{\mu,s}}.
	\end{align*}
	In a similar way, we deduce that $\|y\pa_y N(s)\|_{Y^k_{\mu,s}}$ possesses a same bound.
	Combining all above estimates, we get the first desired result. 
	
For the second result, a direct computation yields
	\begin{align*}
		&\sum_{i+j\leq2}\left\|\left\|\pa_x^i\pa_y^j N(s)\right\|_{L^2_x}\right\|_{L^1_y(y\geq1)}\\
		&\leq  C\sum_{k=0}^2 \sum_{i+j\leq k}\|\pa_x^i\pa_y^j U(s)\|_{L^\infty(1\leq y\leq3)} \sum_{i+j\leq 3-k}\left\|\left\|\pa_x^i\pa_y^j\omega\right\|_{L^2_x}\right\|_{L^1_y(1\leq y\leq3)}\\
		&\quad+C\nu\sum_{i+j\leq3}\left\|\left\|\pa_x^i\pa_y^j\omega\right\|_{L^2_x}\right\|_{L^1_y(1\leq y\leq3)}
		+C\nu\sum_{i+j\leq4}\|\pa_x^i\pa_y^j\omega_c\|_{L^2(1\leq y\leq3)}
		:=I_1+I_2+I_3.
	\end{align*}
Obviously, it holds that
	\begin{align*}
		I_2+I_3\leq C\|\omega(s)\|_{H^3(\frac{7}{8}\leq y\leq4)}
		+C\nu\|\omega_c\|_{H^4(1\leq y\leq3)} .
	\end{align*}
 For $0\leq k\leq2$, we get by Lemma \ref{velocity estimates 1}  that 
	 \begin{align*}
		I_1&\leq C\sum_{k=0}^2 \big(e(s)+\|\omega(s)\|_{H^{1+k}(\frac{7}{8}\leq y\leq4)}+1\big)\|\omega(s)\|_{H^{3-k}(\frac{7}{8}\leq y\leq4)}\\
		&\leq C\big(e(s)+1\big)\|\omega(s)\|_{H^3(\frac{7}{8}\leq y\leq4)}
		+C\|\omega(s)\|_{H^3(\frac{7}{8}\leq y\leq4)}^2.
	\end{align*}
	\end{proof}

We next derive the estimates of the boundary term $B, \widetilde B$ defined in \eqref{eq: eq of omega near the boundary}, \eqref{def of tilde B}.

 \begin{lemma}\label{est of B 1}
	For $0<\mu<\mu_0-\gamma s$, it holds that
	\begin{align*}
		&\sum_{i\leq1}\left\|e^{\eps_0(1+\mu)|\xi|}\xi^i \big( B_\xi(s),\widetilde B_\xi(s)\big)\right\|_{L^1_\xi\cap L^2_\xi}\\
		&\leq C(\mu_0-\mu-\gamma s)^{-\alpha}\Big(\big(E(s)+1\big)^2+e^{\frac{4\eps_0}{\nu}}\sup_{0\leq\tau\leq s}\|(1,x)\omega(\tau)\|^2_{H^3(\frac{7}{8}\leq y\leq4)} \Big)+C\nu^{1/2}s^{-1/2}.
	\end{align*}
	For the  case $i=0$, the factor $(\mu_0-\mu-\gamma s)^{-\alpha}$can be removed.
\end{lemma}

\begin{proof}
We treat $B$ firstly.
 According to the definition of $B$, we utilize Lemma \ref{velocity formula} to get
	\begin{align}\label{decomposition of B xi}
		B_\xi(s)&=\big(\pa_y\Delta_D^{-1}(U\cdot\nabla\omega)\big)_\xi|_{y=0}(s)-\nu(\pa_y+|\xi|)(\omega_c)_\xi|_{y=0}(s)\\
		\nonumber
		&=-\int_0^{1+\mu}e^{-|\xi|z}(U\cdot\nabla\omega)_\xi(s,z)dz
		-\int_{1+\mu}^{+\infty}e^{-|\xi|z}(U\cdot\nabla\omega)_\xi(s,z)dz
		-\nu(\pa_y+|\xi|)(\omega_c)_\xi|_{y=0}(s)\\
		\nonumber
		&=I_1+I_2+I_3.
	\end{align}
	
	We deal with $I_1$ firstly. The following fact
	\begin{align}\label{transfer 1}
		e^{\eps_0(1+\mu)|\xi|}e^{-|\xi|z}
		\leq e^{\eps_0(1+\mu-z)_+|\xi|}
	\end{align}
	gives
	\begin{align*}
		\left|e^{\eps_0(1+\mu)|\xi|}I_1\right|
		\leq \Big(\left\|e^{\eps_0(1+\mu-z)_+|\xi|} (u\pa_x\omega)_\xi(s,z)\right\|_{\mu,s}
		+\left\|e^{\eps_0(1+\mu-z)_+|\xi|} (v\pa_y\omega)_\xi(s,z)\right\|_{\mu,s}\Big).
	\end{align*}
	Thus, we use Lemma \ref{product estimate} and Lemma \ref{velocity estimates 1}, Lemma \ref{est of omega c} to obtain
	\begin{align*}
		\left\|e^{\eps_0(1+\mu)|\xi|}\xi^i I_1\right\|_{L^1_\xi\cap L^2_\xi}
		&\leq \sum_{j+k\leq i}\left\|\sup_{0<y<1+\mu}e^{\eps_0(1+\mu-y)_+|\xi|}|(\pa_x^j u)_\xi(s,y)|\right\|_{L^1_\xi}
		\|\pa_x^{1+k}\omega(s)\|_{Y^1_{\mu,s}\cap Y^2_{\mu,s}}\\
		&+\sum_{j+k\leq i}\left\|\sup_{0<y<1+\mu}e^{\eps_0(1+\mu-y)_+|\xi|}\frac{|(\pa_x^j v)_\xi(s,y)|}{y}\right\|_{L^1_\xi}
		\|\pa_x^k(y\pa_y)\omega(s)\|_{Y^1_{\mu,s}\cap Y^2_{\mu,s}}\\
		&\leq C(\mu_0-\mu-\gamma s)^{-\alpha}\big(e(s)+\|\omega(s)\|_{H^2(1\leq y\leq2)}+1\big)\big(e(s)+1\big).
	\end{align*}
	
For $I_2$, we get by integration by parts that
	\begin{align*}
		I_2&=-\int_{1+\mu}^{+\infty}e^{-|\xi|z}\big(\dv(U\omega)\big)_\xi(s,z)dz\\
		&=-\int_{1+\mu}^{+\infty}e^{-|\xi|z}\big((i\xi)(u\omega)_\xi(s,z)+|\xi|(v\omega)_\xi(s,z)\big)dz
		+e^{-(1+\mu)|\xi|}(v\omega)_\xi(s,1+\mu)\\
		&=-\int_{1+\mu}^{+\infty}e^{-|\xi|z}\big((i\xi)(u\omega)_\xi(s,z)+|\xi|(v\omega)_\xi(s,z)\big)dz
		+e^{-(1+\mu)|\xi|}\int_0^{1+\mu}\pa_z(v\omega)_\xi(s,z)dz\\
		&=-\int_{1+\mu}^{+\infty}e^{-|\xi|z}\big((i\xi)(u\omega)_\xi(s,z)+|\xi|(v\omega)_\xi(s,z)\big)dz\\
		&\quad-e^{-(1+\mu)|\xi|}\int_0^{1+\mu}(\pa_x u\omega)_\xi(s,z)dz
		+e^{-(1+\mu)|\xi|}\int_0^{1+\mu}(v\pa_z\omega)_\xi(s,z)dz.
	\end{align*}
	We then have for $i\leq1$,
	\begin{align*}
		\left|e^{\eps_0(1+\mu)|\xi|}\xi^i I_2\right|
		&\leq C\int_{1+\mu}^{+\infty}e^{-\frac{|\xi|z}{2}}|(U\omega)_\xi(s,z)|dz
		+C\left\|e^{\eps_0(1+\mu-z)_+|\xi|}(\pa_x u\omega)_\xi(s,z)\right\|_{\mu,s}\\
		&\quad+C\left\|e^{\eps_0(1+\mu-z)_+|\xi|}( v \pa_z\omega)_\xi(s,z)\right\|_{\mu,s}.
	\end{align*}
	
	Due to $\|e^{-\frac{|\xi|z}{2}}\|_{L^2_\xi\cap L^\infty_\xi}\leq C$ for $z\geq1+\mu$, we have
	\begin{align*}
		\left\|\int_{1+\mu}^{+\infty}e^{-\frac{|\xi|z}{2}}|(U\omega)_\xi(s,z)|dz\right\|_{L^1_\xi\cap L^2_\xi}
		&\leq C\int_{1+\mu}^{+\infty}\|(U\omega)_\xi(s,z)\|_{L^2_\xi}dz
		=C \int_{1+\mu}^{+\infty}\|(U\omega)(s,z)\|_{L^2_x}dz \\
		&\leq C\| U e^\Psi\chi_0\psi \omega(s)\|_{L^2}
		\leq C\|U(s)\|_{L^\infty}\|e^\Psi\chi_0\psi \omega(s)\|_{L^2}
		\leq C\big(e(s)+1\big)^2,
	\end{align*}
	here we used Plancherel identity and Lemma \ref{velocity estimates 1}. Thus, we obtain
	\begin{align*}
		\left\|e^{\eps_0(1+\mu)|\xi|}\xi^i I_2\right\|_{L^1_\xi\cap L^2_\xi}
		&\leq C\left\|\int_{1+\mu}^{+\infty}e^{-\frac{|\xi|z}{2}}|(U\omega)_\xi(s,z)|dz\right\|_{L^1_\xi\cap L^2_\xi}
		+C\| \pa_x u \omega \|_{Y^1_{\mu,s}\cap Y^2_{\mu,s}}
		+C\|v \pa_y\omega\|_{Y^1_{\mu,s}\cap Y^2_{\mu,s}}\\
		&\leq C\big(e(s)+1\big)^2+C\big(e(s)+1\big)\|\omega(s)\|_{H^2(1\leq y\leq2)}.
		\end{align*}
where we used Lemma \ref{product estimate} and Lemma \ref{velocity estimates 1} to obtain
	\begin{align*}
\|\pa_x u \omega\|_{Y^1_{\mu,s}\cap Y^2_{\mu,s}}
		+\| v \pa_y\omega\|_{Y^1_{\mu,s}\cap Y^2_{\mu,s}}
		\leq& \left\|\sup_{0<y<1+\mu}e^{\eps_0(1+\mu-y)_+|\xi|}|(\pa_x u)_\xi(s,y)|\right\|_{L^1_\xi}
		\|\omega(s)\|_{Y^1_{\mu,s}\cap Y^2_{\mu,s}}\\
		&+\left\|\sup_{0<y<1+\mu}e^{\eps_0(1+\mu-y)_+|\xi|}\frac{|v_\xi(s,y)|}{y}\right\|_{L^1_\xi}
		\|(y\pa_y)\omega(s)\|_{Y^1_{\mu,s}\cap Y^2_{\mu,s}}\\
		\leq& C\big(e(s)+\|\omega(s)\|_{H^2(1\leq y\leq2)}+1\big)\big(e(s)+1\big).
	\end{align*}
	
	For $I_3$, Lemma \ref{est of omega c} implies
	\begin{align*}
		\left\|e^{\eps_0(1+\mu)|\xi|}\xi^i I_3\right\|_{L^1_\xi\cap L^2_\xi}
		\leq C\nu^{1/2}s^{-1/2}.
	\end{align*}
	
 Combining the estimates of $I_1, I_2$ and $I_3$, we obtain
 \begin{align}\label{est of B med}
 	\sum_{i\leq1}\left\|e^{\eps_0(1+\mu)|\xi|}\xi^i  B_\xi(s)\right\|_{L^1_\xi\cap L^2_\xi}
		\leq C(\mu_0-\mu-\gamma s)^{-\alpha}\Big(\big(E(s)+1\big)^2+\|\omega(s)\|^2_{H^3(\frac{7}{8}\leq y\leq 4)} \Big)+C\nu^{1/2}s^{-1/2}.
 \end{align}
 \medskip
 
 Now we turn to treat $\widetilde B$. Recall that $\widetilde B_\xi=i\pa_\xi(B_\xi)-i\nu sgn\xi (\omega-\omega_c)_\xi|_{y=0}$. For the first term $i\pa_\xi(B_\xi)$, taking $\pa_\xi$ on $I_1\sim I_3$ before and using the relation $i\pa_\xi f_\xi=(xf)_\xi$, we derive that $i\pa_\xi(B_\xi)$ the same bound with $B_\xi$ in \eqref{est of B med}. Therefore, we focus on the second term in $\widetilde B_\xi$, that is $i\nu sgn\xi (\omega-\omega_c)_\xi|_{y=0}$. By the solution formula \eqref{eq: solution formula of omega near the boundary}, we have
 \begin{align*}
 	(\chi\omega-\chi\omega_c)_\xi(s,0)
	&=\int_0^{+\infty}\big(H_\xi(s,0,z)+R_\xi(s,0,z)\big)b_\xi(z)dz \\
	\nonumber
	&+\int_0^s\int_0^{+\infty}\big(H_\xi(s-\tau,0,z)+R_\xi(s-\tau,0,z) \big)N_\xi(\tau,z)dzd\tau
	\\
	\nonumber
	&-\int_0^s \big(H_\xi(s-\tau,0,0)+R_\xi(s-\tau,0,0)\big) B_\xi(\tau)d\tau:=J_1+J_2+J_3.
 \end{align*}
 
 For $J_1$, Lemma \ref{lem: est of initial b} yields
 \begin{align*}
 	\nu\sum_{i\leq1}\left\|e^{\eps_0(1+\mu)|\xi|}\xi^i J_1\right\|_{L^1_\xi\cap L^2_\xi}
		\leq C\nu.
 \end{align*}
 
 For $J_2$, by \eqref{def of H xi}, \eqref{def of Gamma xi}, Lemma \ref{prop of R 1} and a direct computation, we have
	\begin{align}\label{est of H med}
		e^{\eps_0(1+\mu)|\xi|}|\xi|^i
		|H_\xi(s-\tau,0,z)|
		\leq 
		\left\{
		\begin{aligned}
			&\frac{C|\xi|^i}{\nu^{1/2}(s-\tau)^{1/2}}e^{\eps_0(1+\mu-z)_+|\xi|},\qquad z<1+\mu,\\
			&C,\qquad z\geq1+\mu,
		\end{aligned}
		\right.
	\end{align}
	and
	\begin{align}\label{est of R med}
		e^{\eps_0(1+\mu)|\xi|}|\xi|^i
		|R_\xi(s-\tau,0,z)|
		\leq 
		\left\{
		\begin{aligned}
			&C\big(\frac{|\xi|^i}{\nu^{1/2}(s-\tau)^{1/2}}+|\xi|^{1+i}\big)e^{\eps_0(1+\mu-z)_+|\xi|},\qquad z<1+\mu,\\
			&C,\qquad z\geq1+\mu.
		\end{aligned}
		\right.
	\end{align}
	Thus, it holds that
	\begin{align*}
		&\nu\sum_{i\leq1}\left\|e^{\eps_0(1+\mu)|\xi|}\xi^i
		J_2\right\|_{L^1_\xi \cap L^2_\xi}
		\leq C\nu^{1/2}\int_0^s(s-\tau)^{-1/2}\sum_{i\leq1}\|\pa_x^i N\|_{Y_{\mu,\tau}^1\cap Y_{\mu,\tau}^2}d\tau\\
		&\qquad+C\nu\int_0^s \sum_{i\leq2}\|\pa_x^i N\|_{Y_{\mu,\tau}^1\cap Y_{\mu,\tau}^2}d\tau
		+C\nu \int_0^s \sum_{i\leq1}\left\|\left\|(\pa_x^i N)_\xi(\tau,z)\right\|_{L^1_z(z\geq1+\mu)}\right\|_{L^1_\xi\cap L^2_\xi}d\tau\\
		&\leq C\nu^{1/2}\int_0^s\big((s-\tau)^{-1/2}+(\mu_0-\mu-\gamma\tau)^{-1}\big)\sum_{i\leq1}\|\pa_x^i N\|_{Y_{\mu_2,\tau}^1\cap Y_{\mu_2,\tau}^2}d\tau\\
		&\qquad+C\nu \int_0^s \sum_{i\leq2}\left\|\left\|\pa_x^i N(\tau,z)\right\|_{L^2_x}\right\|_{L^1(z\geq1+\mu)}d\tau\\
		&\leq C\nu^{1/2}\int_0^s\big((s-\tau)^{-1/2}+(\mu_0-\mu-\gamma\tau)^{-1}\big) \|N(\tau)\|_{W_{\mu_2,\tau}}d\tau,
	\end{align*}
	where we take $\mu_2=\mu+\f12(\mu_0-\mu-\gamma\tau)$ and use Lemma \ref{analytic recovery} in the last but one step.
	
	By Lemma \ref{est of N in new norm} and Lemma \ref{integral computation}, we have
	\begin{align*}
		&\nu\sum_{i\leq1}\left\|e^{\eps_0(1+\mu)|\xi|}\xi^i
		J_2\right\|_{L^1_\xi \cap L^2_\xi}
		\leq C\nu^{1/2}\int_0^s\big((s-\tau)^{-1/2}+(\mu_0-\mu-\gamma\tau)^{-1}\big)\cdot\\
		&\quad \Big((\mu_0-\mu-\gamma\tau)^{-\alpha}(e(\tau)+1)^2+\big((\mu_0-\mu-\gamma\tau)^{-\alpha}+e^{\frac{4\eps_0}{\nu}}\big)\|(1,x)\omega(\tau)\|^2_{H^3(\frac{7}{8}\leq y\leq3)}\Big)d\tau\\
		&\leq C(\mu_0-\mu-\gamma s)^{-\alpha}\Big(\big(E(s)+1\big)^2
		+e^{\frac{4\eps_0}{\nu}}\sup_{0\leq\tau\leq s}\|(1,x)\omega(\tau)\|^2_{H^3(\frac{7}{8}\leq y\leq4)} \Big).
	\end{align*}
	
	For $J_3$, as in $J_2$, we use \eqref{est of H med},\eqref{est of R med}, Lemma \ref{analytic recovery} and take $\mu_2=\mu+\f12(\mu_0-\mu-\gamma\tau)$ to obtain
	\begin{align*}
		&\nu\sum_{i\leq1}\left\|e^{\eps_0(1+\mu)|\xi|}\xi^i
		J_3\right\|_{L^1_\xi \cap L^2_\xi}
		\leq C\nu^{1/2}\int_0^s(s-\tau)^{-1/2}\sum_{i\leq1}\left\|e^{\eps_0(1+\mu)|\xi|}\xi^i B_\xi(\tau)\right\|_{L^1_\xi \cap L^2_\xi}d\tau\\
		&\qquad+C\nu\int_0^s \sum_{i\leq2}\left\|e^{\eps_0(1+\mu)|\xi|}\xi^i B_\xi(\tau)\right\|_{L^1_\xi \cap L^2_\xi}d\tau\\
		&\leq C\nu^{1/2}\int_0^s \big((s-\tau)^{-1/2}+(\mu_0-\mu-\gamma\tau)^{-1}\big)
		\sum_{i\leq1}\left\|e^{\eps_0(1+\mu_2)|\xi|}\xi^i B_\xi(\tau)\right\|_{L^1_\xi \cap L^2_\xi}d\tau.
	\end{align*}
	By \eqref{est of B med} and Lemma \ref{integral computation}, we have
	\begin{align*}
		&\nu\sum_{i\leq1}\left\|e^{\eps_0(1+\mu)|\xi|}\xi^i
		J_3\right\|_{L^1_\xi \cap L^2_\xi}
		\leq C\nu^{1/2}\int_0^s \big((s-\tau)^{-1/2}+(\mu_0-\mu-\gamma\tau)^{-1}\big)\cdot\\
		&\qquad\qquad\qquad\qquad\Big\{ (\mu_0-\mu-\gamma \tau)^{-\alpha}\Big(\big(E(\tau)+1\big)^2+\|\omega(\tau)\|^2_{H^3(\frac{7}{8}\leq y\leq 4)} \Big)+\nu^{1/2}\tau^{-1/2}\Big\} d\tau\\
		&\leq C(\mu_0-\mu-\gamma s)^{-\alpha}\Big(\big(E(s)+1\big)^2
		+\sup_{0\leq\tau\leq s}\|(1,x)\omega(\tau)\|^2_{H^3(\frac{7}{8}\leq y\leq4)} \Big)+C\nu^{1/2}s^{-1/2}.
	\end{align*}
	
	Collecting these estimates together implies the desired result.
\end{proof}

\subsection{Proof of Proposition \ref{prop: Y(t)}}

 \underline{Proof of \eqref{est: Y(t)}}. 
Recalling the definition of functional space $Y_k(t)$, we have
\beno
\|\om\|_{Y_k(t)} =\sup_{\mu<\mu_0-\gamma t} \Big(\sum_{i+j\leq 1}\left\|\pa_x^i(y\pa_y)^j\omega(t)\right\|_{Y^k_{\mu,t}}+(\mu_0-\mu-\gamma t)^\alpha \sum_{ i+j= 2 }\left\|\pa_x^i(y\pa_y)^j\omega(t)\right\|_{Y^k_{\mu,t}} \Big).
\eeno
We treat \eqref{est: Y(t)} firstly. Bringing Lemma \ref{est of N1} into Lemma \ref{lem: om-1} and Lemma \ref{lem: om-3}, we get
\begin{align*}
	&\left\|\int_0^t\int_0^{+\infty}\big(H_\xi(t-s,y,z)+R_\xi(t-s,y,z)\big)N_\xi(s,z)dzds \right\|_{Y_1(t)\cap Y_2(t)} 
	\leq C\sup_{\mu<\mu_0-\gamma t}\int_0^t\|N(s)\|_{W_{\mu,s}}ds\\
	&\quad+C\sup_{\mu<\mu_0-\gamma t}(\mu_0-\mu-\gamma t)^{\alpha}
	\int_0^t \big((\mu_0-\mu-\gamma s)^{-1}+(\mu_0-\mu-\gamma s)^{-1/2}(t-s)^{-1/2}\big)\\
	&\qquad\qquad\cdot\big(\|N(s)\|_{W_{\mu,s}}+\frac{1}{t}\int_0^s\|N(\tau)\|_{W_{\mu_1,\tau}}d\tau\big)ds\\
	&\leq C\sup_{\mu<\mu_0-\gamma t}\int_0^t(\mu_0-\mu-\gamma s)^{-\alpha}ds \cdot \Big(\big(E(t)+1\big)^2+e^{\frac{4\eps_0}{\nu}}\sup_{[0,t]}\|(1,x)\omega(s)\|^2_{H^3(\frac{7}{8}\leq y\leq4)} \Big)\\
	&\quad+C\sup_{\mu<\mu_0-\gamma t}(\mu_0-\mu-\gamma t)^{\alpha}
	\int_0^t \big((\mu_0-\mu-\gamma s)^{-1}+(\mu_0-\mu-\gamma s)^{-1/2}(t-s)^{-1/2}\big)\\
	&\qquad\cdot \Big((\mu_0-\mu-\gamma s)^{-\alpha}\big(e(s)+1\big)^2
		+\big((\mu_0-\mu-\gamma s)^{-\alpha}+e^{\frac{4\eps_0}{\nu}}\big)\sup_{[0,t]}\| (1,x)\omega(s)\|_{H^3(\frac{7}{8}\leq y\leq4)}^2 \Big)ds\\
		&\leq \f{C}{\gamma^{\f12}} \big(E(t)+1\big)^2
		+Ce^{\frac{4\eps_0}{\nu}}\sup_{[0,t]}\| (1,x)\omega(s)\|_{H^3(\frac{7}{8}\leq y\leq4)}^2, 
\end{align*}
where we used Lemma \ref{integral computation} in the last step.

Bringing Lemma \ref{est of B 1} into Lemma \ref{lem: om-2} and Lemma \ref{lem: om-4}, we get
\begin{align*}
	&\left\|\int_0^t \big(H_\xi(t-s,y,0)+R_\xi(t-s,y,0)\big)B_\xi(s)ds\right\|_{Y_1(t)\cap Y_2(t)}\\
	&\leq C\sup_{\mu<\mu_0-\gamma t}\int_0^t \sum_{i\leq1}\left\|e^{\eps_0(1+\mu)|\xi|}\xi^i B_\xi(s)\right\|_{L^1_\xi\cap L^2_\xi}ds+C\sup_{\mu<\mu_0-\gamma t}(\mu_0-\mu-\gamma t)^{\alpha}\int_0^t (\mu_0-\mu-\gamma s)^{-1}\\
	&\qquad\qquad\cdot\Big(\sum_{i\leq1}\left\|e^{\eps_0(1+\mu)|\xi|}\xi^i B_\xi(s)\right\|_{L^1_\xi\cap L^2_\xi}+\frac{1}{t}\int_0^s \sum_{i\leq1}\left\|e^{\eps_0(1+\mu_1)|\xi|}\xi^i B_\xi(\tau)\right\|_{L^1_\xi\cap L^2_\xi}d\tau \Big)ds\\
	&\leq C\sup_{\mu<\mu_0-\gamma t}\int_0^t (\mu_0-\mu-\gamma s)^{-\alpha} \Big(\big(E(s)+1\big)^2+e^{\frac{4\eps_0}{\nu}}\sup_{[0,t]}\|(1,x)\omega(s)\|^2_{H^3(\frac{7}{8}\leq y\leq4)} \Big)ds
	+C\nu^{1/2}t^{1/2}\\
	&+C\sup_{\mu<\mu_0-\gamma t}(\mu_0-\mu-\gamma t)^{\alpha}\int_0^t (\mu_0-\mu-\gamma s)^{-1-\alpha}\Big(\big(E(s)+1\big)^2+e^{\frac{4\eps_0}{\nu}}\sup_{[0,t]}\|(1,x)\omega(s)\|^2_{H^3(\frac{7}{8}\leq y\leq4)} \Big)ds\\
	&+C\nu^{1/2}\sup_{\mu<\mu_0-\gamma t}(\mu_0-\mu-\gamma t)^{\alpha}\int_0^t (\mu_0-\mu-\gamma s)^{-1}s^{-1/2}ds\\
	&\leq \f{C}{\gamma} \big(E(t)+1\big)^2
		+Ce^{\frac{4\eps_0}{\nu}}\sup_{[0,t]}\| (1,x)\omega(s)\|_{H^3(\frac{7}{8}\leq y\leq4)}^2
		+C\nu^{1/2},
\end{align*}
where we used Lemma \ref{integral computation} in the last step.
	
Combining the above estimates with Lemma \ref{lem: est of initial b}, we derive the estimates for $\|(\omega-\omega_c)\|_{Y_1(t)\cap Y_2(t)}$. Using the same argument, we can obtain $\|x(\omega-\omega_c)\|_{Y_1(t)\cap Y_2(t)}$, which admits the same bound as $\|(\omega-\omega_c)\|_{Y_1(t)\cap Y_2(t)}$. With this, we complete the proof of \eqref{est: Y(t)}.

\smallskip

 \underline{Proof of \eqref{est: omega near boundary L infty} }.  For $0<y<\frac{3}{4}$, we utilize \eqref{eq: solution formula of omega near the boundary}, the definition of $H_\xi$ and Lemma \ref{prop of R 1} to obtain
			\begin{align*}
				|\omega_\xi(t,y)|&\leq \int_0^t\int_0^{+\infty}|H_\xi(t-s,y,z)+R_\xi(t-s,y,z)||N_\xi(s,z)|dzds\\
				&\quad+\int_0^t |H_\xi(t-s,y,0)+R_\xi(t-s,y,0)||B_\xi(s)|ds
				+|(\omega_c)_\xi(t,y)| \\
				&\quad+\left|\int_0^{+\infty}\big(H_\xi(t,y,z)+R_\xi(t,y,z)\big)b_\xi(z)dz\right| \\
				&\leq  C\int_0^t\int_0^1\big(\frac{1}{\sqrt{\nu(t-s)}}+\frac{1}{\sqrt{\nu}}+|\xi|\big)|N_\xi(s,z)|dzds
				+C\int_0^t\int_1^3|N_\xi(s,z)|dzds\\
				&\quad+\int_0^t \big(\frac{1}{\sqrt{\nu(t-s)}}+\frac{1}{\sqrt{\nu}}+|\xi|\big)|B_\xi(s)|ds
				+|(\omega_c)_\xi(t,y)|\\
				&\quad+\left|\int_0^{+\infty}\big(H_\xi(t,y,z)+R_\xi(t,y,z)\big)b_\xi(z)dz\right| 
				:=I_1+I_2+I_3+I_4+I_5,
			\end{align*}
			here we used the fact $\operatorname{supp}N\subseteq\{0\leq y\leq 3\}$.
			
			For $I_1$, we get by Lemma \ref{est of N1} that
			\begin{align*}
				\|I_1\|_{L^2_\xi}&\leq C\int_0^t\frac{1}{\sqrt{\nu(t-s)}}\left\|\int_0^1|N_\xi(s,z)|dz\right\|_{L^2_\xi}ds
				+C\int_0^t \left\|\int_0^1|(\pa_x N)_\xi(s,z)|dz\right\|_{L^2_\xi}ds\\
				&\leq C\int_0^t \frac{1}{\sqrt{\nu(t-s)}}\big(e(s)+\|\omega(s)\|_{H^1(1\leq y\leq2)}+1\big)\big( e(s)+1\big) ds\\
				&\qquad+C\int_0^t(\mu_0-\gamma s)^{-\alpha}\Big(\big(e(s)+1\big)^2
		+\| (1,x)\omega(s)\|_{H^2(1\leq y\leq2)}^2 \Big) ds\\
				&\leq C\nu^{-\f12}\Big(\big(E(t)+1\big)^2
		+\sup_{[0,t]}\| (1,x)\omega(s)\|_{H^2(1\leq y\leq2)}^2 \Big) .
			\end{align*}
			
			For $I_2$, by Lemma \ref{est of N1}, we  have
			\begin{align*}
				\|I_2\|_{L^2_\xi}\leq C\Big(\big(E(t)+1\big)^2
		+\sup_{[0,t]}\| (1,x)\omega(s)\|_{H^3(\frac{7}{8}\leq y\leq4)}^2 \Big)  .
			\end{align*}
			
			For $I_3$, we use Lemma \ref{est of B 1} to get
			\begin{align*}
				&\|I_3\|_{L^2_\xi}\leq C\int_0^t\frac{1}{\sqrt{\nu(t-s)}}\left\|e^{\eps_0|\xi|}B_\xi(s)\right\|_{L^2_\xi}ds\\
				&\leq C\int_0^t\frac{1}{\sqrt{\nu(t-s)}}\Big(\big(E(s)+1\big)^2+e^{\frac{4\eps_0}{\nu}}\sup_{0\leq\tau\leq s}\|(1,x)\omega(\tau)\|^2_{H^3(\frac{7}{8}\leq y\leq4)}+\nu^{1/2}s^{-1/2} \Big)  ds\\
				&\leq C\nu^{-\f12}\Big(\big(E(t)+1\big)^2+e^{\frac{4\eps_0}{\nu}}\sup_{[0,t]}\|(1,x)\omega(s)\|^2_{H^3(\frac{7}{8}\leq y\leq4)} \Big) .
			\end{align*}
			
			For $I_4$, we utilize Lemma \ref{est of omega c} to deduce
			\begin{align*}
				\|I_4\|_{L^2_\xi}
				\leq C(\nu t)^{-1/2}
			\end{align*}
			
			For $I_5$, Lemma \ref{lem: est of initial b} implies $\|I_5\|_{L^2_\xi}\leq C.$

			Collecting these estimates together, we derive \eqref{est: omega near boundary L infty}.

\subsection{Proof of some basic estimates}

\begin{proof}[Proof of Lemma \ref{lem: om-1}]
We decompose $H_\xi=e^{-\nu\xi^2t}g(\nu t,y-z)+e^{-\nu\xi^2t}g(\nu t,y+z):=H^-_\xi+H^+_\xi.$ We only prove the inequality for $H^-_\xi$ and the case $H^+_\xi$ is similar.

Firstly, we deal with the first part: $\sum_{i+j\leq1}\left\|\pa_x^i(y\pa_y)^j\int_0^t\int_0^{+\infty}H(t-s,y,z)N(s,z)dzds\right\|_{Y^1_{\mu,t}\cap Y^2_{\mu,t}}$.

 Let $\phi: \mathbb R_+\rightarrow[0,1]$ be a smooth cut-off function such that $\phi(y)=1$ as $0\leq y\leq \frac{1}{2}$ and $\phi(y)=0$ as $y\geq\frac{3}{4}$. 
 \smallskip

	\underline{Case 1: $i=0$, $j=1$.} Using integration by parts, we decompose
	\begin{align*}
		&\quad y\pa_y\int_0^{+\infty}H^-_\xi(t-s,y,z)N_\xi(s,z)dz\\
		&=-y\int_0^{+\infty}\pa_zH^-_\xi(t-s,y,z)N_\xi(s,z)dz\\
		&=-y\int_0^{3y/4}\phi(z/y)\pa_z H^-_\xi(t-s,y,z)N_\xi(s,z)dz
		-\int_{y/2}^{3y/4}\phi'(z/y)H^-_\xi(t-s,y,z)N_\xi(s,z)dz\\
		&\quad+y\int_{y/2}^{1+\mu}\big(1-\phi(z/y)\big)H^-_\xi(t-s,y,z)\pa_z N_\xi(s,z)dz\\
		&\quad+y\int_{1+\mu}^{+\infty}H^-_\xi(t-s,y,z)\pa_z N_\xi(s,z)dz
		:=I_1+I_2+I_3+I_4.
	\end{align*}
	
	Estimate of $I_1$.
	For $0<z<3y/4$, we have $|y\pa_z H^-_\xi|\leq \frac{C}{\sqrt{\nu(t-s)}}e^{-\frac{(y-z)^2}{8\nu(t-s)}}e^{-\nu\xi^2(t-s)}$, which implies
	\begin{align*}
		e^{\eps_0(1+\mu-y)_+|\xi|}|I_1|\leq\int_0^{3y/4}\frac{C}{\sqrt{\nu(t-s)}}e^{-\frac{(y-z)^2}{8\nu(t-s)}}e^{-\nu\xi^2(t-s)}e^{\eps_0(1+\mu-z)_+|\xi|}|N_\xi(s,z)|dz.
	\end{align*}
	For $\eps_0, t$ small enough, we have
	\begin{align}\label{weight transform 1}
		e^{\eps_0(1+\mu)\frac{y^2}{\nu(1+t)}}
		e^{-\frac{(y-z)^2}{100\nu(t-s)}}
		\leq e^{\eps_0(1+\mu)\frac{z^2}{\nu(1+s)}}.
	\end{align}
	Thus,
	\begin{align*}
		\left\|e^{\eps_0(1+\mu-y)_+|\xi|}|I_1|\right\|_{\mu,t}
		&\leq C\int_0^{1+\mu}
		e^{\eps_0(1+\mu)\frac{y^2}{\nu(1+t)}}\int_0^{3y/4}\frac{1}{\sqrt{\nu(t-s)}}e^{-\frac{(y-z)^2}{8\nu(t-s)}}e^{-\nu\xi^2(t-s)}\\
		&\quad\quad\quad e^{\eps_0(1+\mu-z)_+|\xi|}|N_\xi(s,z)|dzdy\\
		&\leq C\int_0^{1+\mu}\int_0^{3y/4}\frac{1}{\sqrt{\nu(t-s)}}e^{-\frac{(y-z)^2}{10\nu(t-s)}}e^{\eps_0(1+\mu)\frac{z^2}{\nu(1+s)}}\\
		&\quad\quad\quad e^{\eps_0(1+\mu-z)_+|\xi|}|N_\xi(s,z)|dzdy\\
		&\leq C\left\|e^{\eps_0(1+\mu-z)_+|\xi|}N_\xi(s)\right\|_{\mu,s}.
	\end{align*}
	
	Estimate of $I_2$. Due to $\|\phi'\|_{L^\infty}\leq C$, we proceed as $I_1$ to obtain
	\begin{align*}
		\left\|e^{\eps_0(1+\mu-y)_+|\xi|}|I_2|\right\|_{\mu,t}
		\leq C\left\|e^{\eps_0(1+\mu-z)_+|\xi|}N_\xi(s)\right\|_{\mu,s}.
	\end{align*}
	
	Estimate of $I_3$.
	We obviously have
	\begin{align*}
		|I_3|\leq C\int_{y/2}^{1+\mu}\left| H^-_\xi(t-s,y,z)z\pa_z N_\xi(s,z)\right|dz.
	\end{align*}
	For $\eps_0$ small enough, we have
	\begin{align}\label{weight transform 2}
		e^{\eps_0(1+\mu-y)_+|\xi|}
		\leq e^{\eps_0(1+\mu-z)_+|\xi|}e^{\eps_0(z-y)_+|\xi|}
		\leq e^{\eps_0(1+\mu-z)_+|\xi|}e^{\frac{(y-z)^2}{100\nu(t-s)}}e^{\frac{\nu\xi^2(t-s)}{10}}.
	\end{align}
	Thus,
	\begin{align*}
		\left\|e^{\eps_0(1+\mu-y)_+|\xi|}|I_3|\right\|_{\mu,t}
		&\leq C\int_0^{1+\mu}e^{\eps_0(1+\mu)\frac{y^2}{\nu(1+t)}}e^{\eps_0(1+\mu-y)_+|\xi|}
		\int_{y/2}^{1+\mu}\left|H^-_\xi(t-s,y,z)z\pa_z N_\xi(s,z)\right|dzdy\\
		&\leq C\int_0^{1+\mu}\int_{y/2}^{1+\mu}
		\frac{1}{\sqrt{\nu(t-s)}}e^{-\frac{(y-z)^2}{10\nu(t-s)}}e^{\eps_0(1+\mu)\frac{z^2}{\nu(1+s)}}\\
		&\quad\quad\quad e^{\eps_0(1+\mu-z)_+|\xi|}\left|z\pa_z N_\xi(s,z)\right|dzdy\\
		&\leq C\left\| e^{\eps_0(1+\mu-z)_+|\xi|}z\pa_z N_\xi(s)\right\|_{\mu,s}.
	\end{align*}
	
	Estimate of $I_4$.
	We utilize \eqref{weight transform 2} to obtain
	\begin{align*}
		e^{\eps_0(1+\mu-y)_+|\xi|}|I_4|
		\leq C\int_{1+\mu}^{+\infty}\frac{1}{\sqrt{\nu(t-s)}}
		e^{-\frac{(y-z)^2}{10\nu(t-s)}}\left| \pa_z N_\xi(s,z)\right|dz,
	\end{align*}
	which implies
	\begin{align*}
		\left\| e^{\eps_0(1+\mu-y)_+|\xi|}I_4\right\|_{\mu,t}
		&\leq C\int_0^{1+\mu}e^{\eps_0(1+\mu)\frac{y^2}{\nu(1+t)}}\int_{1+\mu}^{+\infty}\frac{1}{\sqrt{\nu(t-s)}}e^{-\frac{(y-z)^2}{10\nu(t-s)}}\left|\pa_z N_\xi(s,z)\right|dzdy\\
		&\leq Ce^{\frac{2\eps_0}{\nu}}\left\|\pa_z N_\xi(s,z)\right\|_{L_z^1(z\geq 1+\mu)}.
	\end{align*}
	
	\underline{Case 2: $i=0$, $j=0$.}
	With $\phi$ as above, we use integration by parts to arrive at
	\begin{align*}
		&\quad\int_0^{+\infty}H^-_\xi(t-s,y,z)N_\xi(s,z)dz\\
		&=\int_0^{3y/4}\phi(z/y)H^-_\xi(t-s,y,z)N_\xi(s,z)dz
		+\int_{y/2}^{1+\mu}\big(1-\phi(z/y)\big)H^-_\xi(t-s,y,z)N_\xi(s,z)dz\\
		&\quad+\int_{1+\mu}^{+\infty}H^-_\xi(t-s,y,z)N_\xi(s,z)dz
		:=J_1+J_2+J_3.
	\end{align*}
	
	The estimates of $J_1,J_2, J_3$ are similar with $I_1, I_3, I_4$ in Case 1. Thus, we have
	\begin{align*}
		\left\| e^{\eps_0(1+\mu-y)_+|\xi|}J_1\right\|_{\mu,t}
		+\left\| e^{\eps_0(1+\mu-y)_+|\xi|}J_2\right\|_{\mu,t}
		&\leq C\left\| N(s)\right\|_{\mu,s},\\
		\left\| e^{\eps_0(1+\mu-y)_+|\xi|}J_3\right\|_{\mu,t}
		&\leq Ce^{\frac{2\eps_0}{\nu}}\left\|N_\xi(s,z)\right\|_{L_z^1(z\geq1+\mu)}.
	\end{align*}
	
	\underline{Case 3: $i=1$, $j=0$.}
	This case is similar with Case 2. The only difference is  to replace $N$ with $\pa_x N$. In addition, armed with Plancherel theorem and Minkowski inequality, we have
	\begin{align}\label{transform norm of N}
		\sum_{i+j\leq 1}&\left\|\left\|(\pa_x^i\pa_y^jN)_\xi(s)\right\|_{L^1_y(y\geq1+\mu)}\right\|_{L^1_\xi \cap L^2_\xi}
		\leq C\sum_{i+j\leq1}\left\|\left\|\pa_x^i\pa_y^j N(s)\right\|_{L^2_x}\right\|_{L^1_y(y\geq1+\mu)}\\
		\nonumber
		&\quad +C\sum_{i+j\leq1}\|(1+|\xi|^2)^{-1/2}\|_{L^2_\xi}\left\|(1+|\xi|^2)^{1/2}\|(\pa_x^i\pa_y^j N)_\xi(s)\|_{L^1_y(y\geq1+\mu)}\right\|_{L^2_\xi}\\
		\nonumber
		&\leq C\sum_{i+j\leq2}\left\|\left\|\pa_x^i\pa_y^jN(s)\right\|_{L^2_x}\right\|_{L^1_y(y\geq1+\mu)}.
	\end{align}

Combining all above estimates, we derive  that for $0<\mu<\mu_0-\gamma s$, $i+j\leq1$, 
	\begin{align*} 
		&\left\|\pa_x^i(y\pa_y)^j\int_0^{+\infty}H(t-s,y,z)N(s,z)dz\right\|_{Y^1_{\mu,t}\cap Y^2_{\mu,t}}\\
		&\leq C\|\pa_x^i(y\pa_y)^jN(s)\|_{Y^1_{\mu,t}\cap Y^2_{\mu,t}}
		+C\|N(s)\|_{Y^1_{\mu,t}\cap Y^2_{\mu,t}}
		+Ce^{\frac{2\eps_0}{\nu}}\sum_{i+j\leq2}\left\|\|\pa_x^i\pa_y^j N(s)\|_{L^2_x}\right\|_{L^1_y(y\geq1+\mu)},
	\end{align*}
 which implies
	\begin{align}\label{Step 1 in est of H}
 &\sum_{i+j\leq1}\left\|\pa_x^i(y\pa_y)^j\int_0^t\int_0^{+\infty}H(t-s,y,z)N(s,z)dzds\right\|_{Y^1_{\mu,t}\cap Y^2_{\mu,t}}\leq C\int_0^t \|N(s)\|_{W_{\mu,s}} ds.
	\end{align}

Next, we discuss the case $i+j=2$. Here, we change "analytical radius" $\mu$ to $\mu_1$ to overcome the loss of derivative. More precisely, we utilize the first inequality in Lemma \ref{analytic recovery} for the case $j\leq1$, and the second inequality in Lemma \ref{analytic recovery} for $j=2$. Thus, we obtain
\begin{align*}
	&\sum_{i+j=2} \left\|\pa_x^i(y\pa_y)^j\int_0^t\int_0^{+\infty}H(t-s,y,z)N(s,z)dzds \right\|_{Y^1_{\mu,t}\cap Y^2_{\mu,t}}\\
	&\leq C\int_0^t \big((\mu_0-\mu-\gamma s)^{-1}+(\mu_0-\mu-\gamma s)^{-\frac{1}{2}}(t-s)^{-\frac{1}{2}}\big)\\
	&\qquad\cdot\sum_{i+j\leq1}\left\|\pa_x^i(y\pa_y)^j \int_0^{+\infty}H(t-s,y,z)N(s,z)dz \right\|_{Y^1_{\mu_1,s}\cap Y^2_{\mu_1,s}}ds\\
	&\leq C\int_0^t \big((\mu_0-\mu-\gamma s)^{-1}+(\mu_0-\mu-\gamma s)^{-\frac{1}{2}}(t-s)^{-\frac{1}{2}}\big)
	\|N(s)\|_{W_{\mu_1,s}}ds.
\end{align*}

Combing all above estimates, we obtain the desired results.
\end{proof}

\begin{proof}[Proof of Lemma \ref{lem: om-2}]
 The lemma follows directly from Lemma \ref{analytic recovery} and the following bounds
\begin{align*}
		e^{\eps_0(1+\mu)\frac{y^2}{\nu(1+t)}}e^{-\frac{y^2}{8\nu(t-s)}}
		&\leq C, \quad\text{for}\quad \eps_0\ll1,\\
		\left\|(y\pa_y)^j\big(\frac{1}{\sqrt{\nu(t-s)}}e^{-\frac{y^2}{8\nu(t-s)}}\big)\right\|_{L^1_y}&\leq C \quad\text{for}\quad j=0,1,2.
\end{align*}
\end{proof}

To prove Lemma \ref{lem: om-3} and Lemma \ref{lem: om-4}, we divide the half plane into $(0,\sqrt{\nu t})$ and $(\sqrt{\nu t}, \infty)$.  We introduce a smooth cut-off function $\phi_r:\mathbb R_+\rightarrow[0,1]$ which satisfies $\phi_r(y)=1$ if $y\leq r$, $\phi_r(y)=0$ if $y\geq2r$, and $\phi^c_r=1-\phi_r$. We first derive the estimates on $(\sqrt{\nu t}, \infty)$.
\begin{lemma}\label{est of RN1}
For $\mu<\mu_0-\gamma t$ and $\mu_1=\mu+\frac{1}{2}(\mu_0-\mu-\gamma s)$, we have
	\begin{align*}
		&\sum_{i+j\leq1}\left\|\phi^c_{\sqrt{\nu t}}(y)\pa_x^i(y\pa_y)^j\int_0^t\int_0^{+\infty}R(t-s,y,z)N(s,z)dzds\right\|_{Y^1_{\mu,t}\cap Y^2_{\mu,t}}\\
		&\leq \frac{C}{t}\int_0^t\int_0^s
		\big(\sum_{i\leq1}\left\|\pa_x^i N(\tau)\right\|_{Y^1_{\mu,\tau}\cap Y^2_{\mu,\tau}}
		+e^{\frac{2\eps_0}{\nu}}\sum_{i\leq2}\left\|\left\|\pa_x^iN(\tau)\right\|_{L^2_x}\right\|_{L^1_y(y\geq1+\mu)}\big)d\tau ds,
	\end{align*}
	and
	\begin{align*}
		&\sum_{i+j=2}\left\|\phi^c_{\sqrt{\nu t}}(y)\pa_x^i(y\pa_y)^j\int_0^t\int_0^{+\infty}R(t-s,y,z)N(s,z)dzds\right\|_{Y^1_{\mu,t}\cap Y^2_{\mu,t}}\\
		&\leq \frac{C}{t}\int_0^t(\mu_0-\mu-\gamma s)^{-1}
		\int_0^s \big(\sum_{i\leq1}\left\|\pa_x^i N(\tau)\right\|_{Y^1_{\mu_1,\tau}\cap Y^2_{\mu_1,\tau}}
		+e^{\frac{2\eps_0}{\nu}}\sum_{i\leq2}\left\|\left\|\pa_x^iN(\tau)\right\|_{L^2_x}\right\|_{L^1_y(y\geq1+\mu_1)}\big)d\tau ds.
	\end{align*}
\end{lemma}

\begin{proof}
	For $i+j\leq1$, we take advantage of Lemma \ref{prop of R 1} to have
	\begin{align*}
		&\left\|\phi^c_{\sqrt{\nu t}}(y)\pa_x^i(y\pa_y)^j\int_0^t\int_0^{+\infty}R(t-s,y,z)N(s,z)dzds\right\|_{Y^1_{\mu, t}\cap Y^2_{\mu, t}}\\
		&=\Bigg\|\Bigg\|\phi^c_{\sqrt{\nu t}}e^{\eps_0(1+\mu)\frac{y^2}{\nu(1+t)}}e^{\eps_0(1+\mu-y)_+|\xi|}\nu\int_0^t\int_0^s\int_0^{+\infty}(-\xi^2+\xi\pa_y)(y\pa_y)^j\\
		&\quad\quad
		\big(e^{-\nu(s-\tau)\xi^2}\frac{1}{\sqrt{\nu(t-s)}}e^{-\frac{(y+z)^2}{4\nu(t-s)}}\big)\xi^i N_\xi(\tau,z)dz d\tau ds \Bigg\|_{L^1_y(0,1+\mu)}\Bigg\|_{L^1_\xi\cap L^2_\xi}
		:=\|\|A\|_{L^1_y(0,1+\mu)}\|_{L^1_\xi\cap L^2_\xi}.
	\end{align*}
	We use \eqref{weight transform 1}, \eqref{weight transform 2} and the following inequality which holds for $j=0,1,2$ and $y\geq\sqrt{\nu t}$
	\begin{align*}
		\left|(-\nu\xi^2+\nu|\xi|\pa_y)(y\pa_y)^j\big(e^{-\nu(s-\tau)\xi^2} e^{-\frac{(y+z)^2}{4\nu(s-\tau)}}\big)\right|
		\leq \frac{C}{t}e^{-\nu(s-\tau)\xi^2/2}
		e^{-\frac{(y+z)^2}{5\nu(s-\tau)}},
	\end{align*}
	to obtain
\begin{align*}
	A\leq& \frac{C}{t}\int_0^t\int_0^s\int_0^{1+\mu}\frac{1}{\sqrt{\nu(s-\tau)}}e^{-\frac{(y+z)^2}{10\nu(s-\tau)}}e^{\eps_0(1+\mu)\frac{z^2}{\nu(1+s)}}e^{\eps_0(1+\mu-z)_+|\xi|}\left|\xi^i N_\xi(\tau,z)\right|dzd\tau ds\\
	&+\frac{C}{t}\int_0^t\int_0^s\int_{1+\mu}^{+\infty}e^{\frac{2\eps_0}{\nu}}\frac{1}{\sqrt{\nu(s-\tau)}}e^{-\frac{(y+z)^2}{10\nu(s-\tau)}}\left|\xi^i N_\xi(\tau,z)\right|dzd\tau ds.
\end{align*}
Thus,
\begin{align*}
	\|A\|_{L^1_y(0,1+\mu)}
	\leq& \frac{C}{t}\int_0^t\int_0^s\int_0^{1+\mu}
	e^{\eps_0(1+\mu)\frac{z^2}{\nu(1+s)}}e^{\eps_0(1+\mu-z)_+|\xi|}\left|\xi^iN_\xi(\tau,z)\right|dzd\tau ds\\
	&+\frac{C}{t}\int_0^t\int_0^s\int_{1+\mu}^{+\infty}e^{\frac{2\eps_0}{\nu}}\left|\xi^iN_\xi(\tau,z)\right|dzd\tau ds.
\end{align*}

Now we utilize \eqref{transform norm of N} to obtain
\begin{align*}
	\|\|A\|_{L^1_y(0,1+\mu)}\|_{L^1_\xi\cap L^2_\xi}&\leq \frac{C}{t}\int_0^t\int_0^s\big(\|\pa_x^iN(\tau)\|_{Y^1_{\mu,\tau}\cap Y^2_{\mu,\tau}}
	+e^{\frac{2\eps_0}{\nu}}
	\sum_{i\leq1}\left\|\left\|\xi^iN_\xi(\tau)\right\|_{L^1_y(y\geq1+\mu)}\right\|_{L^1_\xi\cap L^2_\xi}\big)d\tau ds\\
	&\leq \frac{C}{t}\int_0^t\int_0^s
		\big(\left\|\pa_x^i N(\tau)\right\|_{Y^1_{\mu,\tau}\cap Y^2_{\mu,\tau}}
		+e^{\frac{2\eps_0}{\nu}}\sum_{i\leq2}\left\|\left\|\pa_x^iN(\tau)\right\|_{L^2_x}\right\|_{L^1_y(y\geq1+\mu)}\big)d\tau ds.
\end{align*}
Thus, we obtain the first inequality. The second inequality is obtained by Lemma \ref{analytic recovery}.
 \end{proof}
\medskip

By the same argument, we have

\begin{lemma}\label{est of RB1}
For $\mu<\mu_0-\gamma t$ and $\mu_1=\mu+\frac{1}{2}(\mu_0-\mu-\gamma s)$, we have
	\begin{align*}
		\sum_{i+j\leq1}\left\|\phi^c_{\sqrt{\nu t}}\pa_x^i(y\pa_y)^j\int_0^t R(t-s,y,0)B(s)ds\right\|_{Y^1_{\mu, t}\cap Y^2_{\mu, t}}
		\leq \frac{C}{t}\int_0^t\int_0^s
		\sum_{i\leq1}\left\|e^{\eps_0(1+\mu)|\xi|}\xi^iB_\xi(\tau)\right\|_{L^1_\xi\cap L^2_\xi}d\tau ds,
	\end{align*}
	and
	\begin{align*}
		&\sum_{i+j=2}\left\|\phi^c_{\sqrt{\nu t}}\pa_x^i(y\pa_y)^j\int_0^t R(t-s,y,0)B(s)ds\right\|_{Y^1_{\mu, t}\cap Y^2_{\mu, t}}\\ &\leq \frac{C}{t}\int_0^t(\mu_0-\mu-\gamma s)^{-1}\int_0^s
		\sum_{i\leq1}\left\|e^{\eps_0(1+\mu_1)|\xi|}\xi^iB_\xi(s)\right\|_{L^1_\xi\cap L^2_\xi}d\tau ds.
	\end{align*}
\end{lemma}

Next, we dreive the estimates on $(0, \sqrt{\nu t})$.
\begin{lemma}\label{est of RN2}
For $\mu<\mu_0-\gamma t$ and $\mu_1=\mu+\frac{1}{2}(\mu_0-\mu-\gamma s)$, we have
	\begin{align*}
		&\sum_{i+j\leq1}\left\|\phi_{\sqrt{\nu t}}(y)\pa_x^i(y\pa_y)^j\int_0^t\int_0^{+\infty}R(t-s,y,z)N(s,z)dzds\right\|_{Y^1_{\mu, t}\cap Y^2_{\mu, t}}\\
		&\leq C\int_0^t\sum_{i\leq1}\left\|\pa_x^iN(s)\right\|_{Y^1_{\mu, t}\cap Y^2_{\mu, t}}
		+\sum_{i\leq2}\left\|\left\|\pa_x^i N(s)\right\|_{L^2_x}\right\|_{L^1_y(y\geq1+\mu)}ds,
	\end{align*}
	and
	\begin{align*}
		&\sum_{i+j=2}\left\|\phi_{\sqrt{\nu t}}(y)\pa_x^i(y\pa_y)^j\int_0^t\int_0^{+\infty}R(t-s,y,z)N(s,z)dzds\right\|_{Y^1_{\mu, t}\cap Y^2_{\mu, t}}\\
		&\leq C\int_0^t(\mu_0-\mu-\gamma s)^{-1}\big(\sum_{i\leq1}\left\|\pa_x^iN(s)\right\|_{Y^1_{\mu_1,s}\cap Y^2_{\mu_1,s}}
		+\sum_{i\leq2}\left\|\left\|\pa_x^i N(s)\right\|_{L^2_x}\right\|_{L^1_y(y\geq1+\mu_1)}\big)ds.
	\end{align*}
\end{lemma}

\begin{proof}
	For $\eps_0$ small enough, we have
	\begin{align*}
		e^{\eps_0(1+\mu-y)_+|\xi|}
		\leq e^{\eps_0(1+\mu-z)_+|\xi|}e^{\eps_0(z-y)_+|\xi|}
		\leq Ce^{\eps_0(1+\mu-z)_+|\xi|}\cdot
		\left\{
		\begin{aligned}
			&e^{\frac{\theta_0 }{4}a(y+z)},\\
			&e^{\frac{\theta_0}{4}\frac{(y+z)^2}{\nu(t-s)}}e^{\nu\xi^2(t-s)/8},
		\end{aligned}
		\right.
	\end{align*}
	here $a=|\xi|+\frac{1}{\sqrt{\nu}}$.
	Then we use Lemma \ref{prop of R 1} to obtain
	\begin{align*}
		e^{\eps_0(1+\mu-y)_+|\xi|}\left|(y\pa_y)^kR_\xi(t-s,y,z)\right|
		\leq Ce^{\eps_0(1+\mu-z)_+|\xi|}
		\big(ae^{-\frac{\theta_0}{4}a(y+z)}+\frac{1}{\sqrt{\nu(t-s)}}e^{-\frac{\theta_0}{4}\frac{(y+z)^2}{\nu(t-s)}}\big).
	\end{align*}
	Thus,
	\begin{align*}
		\left\|e^{\eps_0(1+\mu-y)_+|\xi|}(y\pa_y)^kR_\xi(t-s,y,z)\right\|_{L^1_y}
		\leq Ce^{\eps_0(1+\mu-z)_+|\xi|}.
	\end{align*}
Using the fact $\phi_{\sqrt{\nu t}}(y)e^{\eps_0(1+\mu)\frac{y^2}{\nu(1+t)}}\leq C$ to obtain
	\begin{align*}
		&\quad\sum_{i+j\leq1}\left\|\phi_{\sqrt{\nu t}}(y)\pa_x^i(y\pa_y)^j\int_0^t\int_0^{+\infty}R(t-s,y,z)N(s,z)dzds\right\|_{Y^1_{\mu, t}\cap Y^2_{\mu, t}}\\
		&\leq C\int_0^t\left\|\int_0^{+\infty}e^{\eps_0(1+\mu-z)_+|\xi|}\sum_{i\leq 1}\left|\xi^iN_\xi(s,z)\right|dz\right\|_{L^1_\xi\cap L^2_\xi} ds\\
		&\leq C\int_0^t\big(\sum_{i\leq1}\left\|\pa_x^iN(s)\right\|_{Y^1_{\mu, s}\cap Y^2_{\mu, s}}
		+\sum_{i\leq1}\left\|\left\|\xi^iN_\xi(s)\right\|_{L^1_y(y\geq1+\mu)}\right\|_{L^1_\xi\cap L^2_\xi}\big)ds.
	\end{align*}
	Finally, we use \eqref{transform norm of N} to obtain the first inequality of this lemma. The second inequality is obtained by Lemma \ref{analytic recovery}.
\end{proof}

By the same argument, we have
\begin{lemma}\label{est of RB2}
For $\mu<\mu_0-\gamma t$ and $\mu_1=\mu+\frac{1}{2}(\mu_0-\mu-\gamma s)$, we have
	\begin{align*}
		\sum_{i+j\leq1}\left\|\phi_{\sqrt{\nu t}}\pa_x^i(y\pa_y)^j\int_0^t R(t-s,y,0)B(s)ds\right\|_{Y^1_{\mu, t}\cap Y^2_{\mu, t}}
		\leq C\int_0^t
		\sum_{i\leq1}\left\|e^{\eps_0(1+\mu)|\xi|}\xi^iB_\xi(s)\right\|_{L^1_\xi\cap L^2_\xi}ds,
	\end{align*}
	and
	\begin{align*}
		&\sum_{i+j=2}\left\|\phi_{\sqrt{\nu t}}\pa_x^i(y\pa_y)^j\int_0^t R(t-s,y,0)B(s)ds\right\|_{Y^1_{\mu, t}\cap Y^2_{\mu, t}} \\
		&\leq C\int_0^t(\mu_0-\mu-\gamma s)^{-1}
		\sum_{i\leq1}\left\|e^{\eps_0(1+\mu_1)|\xi|}\xi^iB_\xi(s)\right\|_{L^1_\xi\cap L^2_\xi}ds.
	\end{align*}
\end{lemma}

Combining Lemmas \ref{est of RN1}–\ref{est of RB2}, we obtain Lemmas \ref{lem: om-3} and \ref{lem: om-4}.

\section{Estimates of the vorticity away from the boundary}\label{Estimates away from the Boundary}

In this section, we present the estimates for the remaining parts of the energy $e(t)$: $\|e^\Psi\chi_0\psi\omega(t)\|_{L^2}$ and $\|\chi_0\omega(t)\|_{L^p}$. 

\subsection{Estimate of $\|e^\Psi\chi_0\psi\omega\|_{L^2}$}

In this subsection, we prove Proposition \ref{weighted est 1}.

\begin{proof}
	Multiplying $\chi_0$ on both sides of Navier-Stokes system \eqref{eq: NS vorticity} leads to
	\begin{align}\label{eq: chi0 NS}
		\pa_t(\chi_0\omega)+U\cdot\nabla(\chi_0\omega)-\nu\Delta(\chi_0\omega)
		=v\chi_0'\omega-2\nu\chi_0'\pa_y\omega-\nu\chi_0''\omega.
	\end{align}
	We take $L^2$ inner product with $e^{2\Psi}\psi^2\chi_0\omega$ to obtain
	\begin{align*}
		&\frac{1}{2}\frac{d}{dt}\left\|e^\Psi\chi_0\psi\omega\right\|^2_{L^2}
		+\frac{20\gamma\eps_0}{\nu}\left\|e^\Psi\chi_0\psi\omega\right\|^2_{L^2(y_1(t)\leq y\leq y_2(t))}
		 -\nu\left\langle\Delta(\chi_0\omega),e^{2\Psi}\psi^2\chi_0\omega\right\rangle\\
		&=-\left\langle U\cdot\nabla(\chi_0\omega),e^{2\Psi}\psi^2\chi_0\omega\right\rangle
		+\langle v\chi_0'\omega,\psi^2\chi_0\omega\rangle
		-2\nu\langle\chi_0'\pa_y\omega,\psi^2\chi_0\omega\rangle
		-\nu\langle\chi_0''\omega,\psi^2\chi_0\omega\rangle:=\sum_{i=1}^4 I_i,
	\end{align*}
	here we notice that $\Psi=0$ on $\operatorname{supp}\chi_0'$.
	
	For the dissipative term, we take advantage of integration by parts to obtain
	\begin{align}\label{dissipative term 1}
		-\nu\left\langle\Delta(\chi_0\omega),e^{2\Psi}\psi^2\chi_0\omega\right\rangle
		&=\nu\left\|e^{2\Psi}\psi^2\nabla(\chi_0\omega)\right\|^2_{L^2}
		+\nu\left\langle\nabla(\chi_0\omega),\nabla(e^{2\Psi}\psi^2)\chi_0\omega\right\rangle.
	\end{align}
	We notice that when $1/4\leq y$, it holds
	\begin{align}\label{property of weight}
		|\nabla(e^{2\Psi}\psi^2)|
		&\leq 2|\pa_y\Psi|e^{2\Psi}\psi^2
		+2e^{2\Psi}\psi|\nabla\psi|\leq \frac{C}{\nu}e^{2\Psi}\psi^2\mathbb I_{(y_1(t)\leq y\leq y_2(t))}
		+Ce^{2\Psi}\psi^2,
	\end{align}
	which implies that the right hand side of \eqref{dissipative term 1} is larger than
	\begin{align*}
		&\nu\left\|e^{2\Psi}\psi^2\nabla(\chi_0\omega)\right\|^2_{L^2}
		-C\int_{y_1(t)\leq y\leq y_2(t)}e^{2\Psi}\psi^2|\nabla(\chi_0\omega)|\cdot|\chi_0\omega|dxdy\\
		&\quad-C\nu\int_{\mathbb R^2_+}e^{2\Psi}\psi^2|\nabla(\chi_0\omega)|\cdot|\chi_0\omega|dxdy\\
		&\geq \frac{9\nu}{10}\left\|e^{2\Psi}\psi^2\nabla(\chi_0\omega)\right\|^2_{L^2}
		-\frac{C}{\nu}\left\|e^\Psi\chi_0\psi\omega\right\|^2_{L^2(y_1(t)\leq y\leq y_2(t))}
		-C\nu\left\|e^\Psi\chi_0\psi\omega\right\|^2_{L^2}.
	\end{align*}
	
	For $I_1$, we utilize integration by parts, $\dv U=0$ and \eqref{property of weight} to obtain
	\begin{align*}
		|I_1|&\leq \f12\int_{\mathbb R^2_+}|U|\cdot|\nabla(e^{2\Psi}\psi^2) |(\chi_0\omega)^2 dxdy\\
		&\leq \frac{C\|U\|_{L^\infty}}{\nu}\int_{y_1(t)\leq y\leq y_2(t)}e^{2\Psi}\psi^2(\chi_0\omega)^2dxdy
		+C\|U\|_{L^\infty}\int_{\mathbb R^2_+}e^{2\Psi}\psi^2(\chi_0\omega)^2dxdy\\
		&\leq \frac{C\big(e(t)+1\big)}{\nu}\left\|e^\Psi\chi_0\psi\omega\right\|^2_{L^2(y_1(t)\leq y\leq y_2(t))}
		+C\big(e(t)+1\big)^3,
	\end{align*}
	here we used Lemma \ref{velocity estimates 1} in the last step.
	
	For $I_2\sim I_4$, Sobolev embedding and Lemma \ref{velocity estimates 1}, Lemma \ref{est of omega c} give rise to
	\begin{align*}
		&\quad |I_2|+|I_3|+|I_4|\\
		&\leq C\|v\|_{L^\infty}\| \omega\|_{L^2(\frac{1}{4}\leq y\leq\frac{3}{8})}\|e^\Psi\psi\chi_0\omega\|_{L^2}
		+C\nu\| \omega\|_{H^1(\frac{1}{4}\leq y\leq\frac{3}{8})}\|e^\Psi\psi\chi_0\omega\|_{L^2}\\
		&\leq C\big(e(t)+1\big)^2\sum_{j\leq1}\| (y\pa_y)^j(\omega-\omega_c+\omega_c)\|_{Y^1_{\mu,t}}
		+C\nu e(t)\sum_{j\leq2}\| (y\pa_y)^j(\omega-\omega_c+\omega_c)\|_{Y^1_{\mu,t}}\\
		&\leq C\big(e(t)+1\big)^3
		+C\nu(\mu_0-\mu-\gamma t)^{-\alpha}e(t)^2.
	\end{align*}

Collecting these estimates together and integrating from $0$ to $t$, we use the definition of $E(t)$ and choose a suitable $\gamma$ to conclude the proof.
\end{proof}

\subsection{Estimate of $\|\chi_0\omega(t)\|_{L^p}$}
This subsection is devoted to proving Proposition \ref{L4 est of omega 1} and obtain the estimate of $\|\chi_0\omega\|_{L^p}$, which is used to control $\|U\|_{L^\infty}$.

\begin{proof}
Taking inner product with $\chi^2_0\omega |\chi_0\omega|^{p-2}$ and integrating by parts give
	\begin{align*}
		& \frac{d}{dt}\|\chi_0\omega\|^p_{L^p}
		+ \nu\int_{\mathbb R^2_+}|\nabla(\chi_0\omega)|^2 |\chi_0\omega|^{p-2} dxdy\\
		& \leq  C\int_{\mathbb R^2_+} | \chi_0' v\omega |\cdot |\chi_0\omega|^{p-1}  +
		 C\nu   \int_{\mathbb R^2_+} | \chi_0'\pa_y\omega|\cdot |\chi_0\omega|^{p-1} 		+C\nu  \int_{\mathbb R^2_+}  |\chi_0''\omega|\cdot  |\chi_0\omega|^{p-1} \\
&\leq C\|v\|_{L^\infty}\|\omega\|_{L^p(\frac{1}{4}\leq y\leq \frac{3}{8})}\|\chi_0\omega\|_{L^p}^{p-1}
		+C\nu\sum_{j\leq1}\|\pa_y^j\omega\|_{L^p(\frac{1}{4}\leq y\leq \frac{3}{8})}\|\chi_0\omega\|^{p-1}_{L^p}\\
		&\leq C\big(e(t)+1\big)^p\sum_{j\leq1}\|(y\pa_y)^j(\omega-\omega_c+\omega_c)\|_{Y^1_{\mu,t}}+C\nu\sum_{j\leq 2}\|(y\pa_y)^j(\omega-\omega_c+\omega_c)\|_{Y^1_{\mu,t}}\|\chi_0\omega\|^{p-1}_{L^p}\\
		&\leq 
		C\big(e(t)+1\big)^{p+1}+C\nu(\mu_0-\mu-\gamma t)^{-\alpha}\big(e(t)+1\big)^p,
	\end{align*}
	where we used Lemma \ref{est of omega c} in the last step. Integrating over $0\leq s\leq t$, we obtain
	\begin{align*}
		\sup_{[0,t]}\|\chi_0\omega\|_{L^p}^p
		\leq C\|\chi_0\omega_0\|_{L^p}^p
		+C(t+\frac{\nu}{\gamma})\big(E(t)+1\big)^{p+1}.
	\end{align*}
	\end{proof}

\subsection{Estimates  in a strip $\frac{7}{8}\leq y\leq 4$ }\label{Sobolev estimates in a strip}

In this subsection, we prove Proposition \ref{Sobolev estimate} through the following lemmas. We only prove for $\omega$, since $x\omega$ is estimated in a same way.

\begin{lemma}\label{Sobolev est 1}
	There exists $T_0$ small enough such that for $0\leq t\leq T_0$,
	\begin{align*}
		\sup_{[0,t]}\|\omega\|^2_{L^2(\frac{1}{2}+\frac{3}{32}\leq y\leq 5-\frac{1}{6})}
		+\nu\int_0^t\|\nabla\omega\|^2_{L^2(\frac{1}{2}+\frac{3}{32}\leq y\leq 5-\frac{1}{6})}ds
		\leq C\nu^8 t \big(E(t)+1\big)^3 e^{-\frac{10\eps_0}{\nu}}.
	\end{align*}
\end{lemma}

\begin{proof}
	We choose a smooth function $\eta_1(y)$ satisfying
	\begin{align*}
		\eta_1(y)=
		\left\{
		\begin{aligned}
			&1,\quad \frac{1}{2}+\frac{3}{32}\leq y\leq 5-\frac{1}{6},\\
			&0,\quad y\leq\frac{1}{2}\quad \text{or}\quad y\geq5.
		\end{aligned}
		\right.
	\end{align*}
	Taking $L^2$ inner product with $\eta_1^2\omega$ on both sides of \eqref{eq: NS vorticity} and integrating over $0\leq s\leq t$, we arrive at
	\begin{align*}
		\int_0^t\int_{\mathbb R^2_+}\pa_t\omega\cdot\eta_1^2\omega dxdyds
		+\int_0^t\int_{\mathbb R^2_+}(U\cdot\nabla\omega)\cdot\eta_1^2\omega dxdyds
		=\nu\int_0^t\int_{\mathbb R^2_+}\Delta\omega\cdot\eta_1^2\omega dxdyds.
	\end{align*}
	Integrating by parts, we utilize the fact $\omega|_{t=0}=0$ on $\operatorname{supp}\eta_1$ to obtain
	\begin{align*}
		&\frac{1}{2}\int_{\mathbb R^2_+}\omega(\cdot,t)^2\eta_1^2 dxdy
		+\nu\int_0^t\int_{\mathbb R^2_+}\eta_1^2|\nabla\omega|^2 dxdyds\\
		&=-2\nu\int_0^t\int_{\mathbb R^2_+}\eta_1\pa_y\omega\cdot\eta_1'\omega dxdyds
		+\int_0^t\int_{\mathbb R^2_+}\eta_1\cdot\eta_1'\cdot U\omega^2 dxdyds,
	\end{align*}
	which by Cauchy inequality gives
	\begin{align*}
		\sup_{[0,t]}\|\eta_1\omega\|^2_{L^2}
		+\nu\int_0^t\|\eta_1\nabla\omega\|^2_{L^2}ds
		&\leq C\nu\int_0^t\|\omega\|^2_{L^2(\frac{1}{2}\leq y\leq5)}ds
		+C\int_0^t\|U\|_{L^\infty}\|\omega\|^2_{L^2(\frac{1}{2}\leq y\leq5)}ds\\
		&\leq C\int_0^t\big(e(s)+1\big)\|\omega\|^2_{L^2(\frac{1}{2}\leq y\leq5)}ds\\
		&\leq C\big(E(t)+1\big)\int_0^t\|\omega\|^2_{L^2(\frac{1}{2}\leq y\leq5)}ds,
	\end{align*}
	here we used Lemma \ref{velocity estimates 1}.
	
Due to the construction of $\Psi$, the following fact holds for $\eps_0$ small enough
	\begin{align}\label{control of nu negative power}
		\nu^{-8}e^{\frac{10\eps_0}{\nu}}
		\leq Ce^\Psi,\quad
		\text{for}\quad
		y\in(\f12,5),
	\end{align}
	which leads to
	\begin{align*}
		\|\omega\|^2_{L^2(\frac{1}{2}\leq y\leq5)}
		\leq C\nu^8\|e^\Psi\chi_0\psi\omega\|_{L^2}^2 e^{-\frac{10\eps_0}{\nu}}
		\leq C\nu^8 e(s)^2 e^{-\frac{10\eps_0}{\nu}}.
	\end{align*}
	
	Combining the estimates together, we have
	\begin{align*}
		\sup_{[0,t]}\|\eta_1\omega\|^2_{L^2}
		+\nu\int_0^t\|\eta_1\nabla\omega\|^2_{L^2}ds
		&\leq C\big(E(t)+1\big)\nu^8 e^{-\frac{10\eps_0}{\nu}}\int_0^t e(s)^2ds\\
		&\leq C\nu^8 t \big(E(t)+1\big)^3 e^{-\frac{10\eps_0}{\nu}},
	\end{align*}
	and conclude the proof.
\end{proof}

\begin{lemma}\label{Sobolev est 2}
	There exists $T_0$ small enough such that for $0\leq t\leq T_0$,
	\begin{align*}
		\sup_{[0,t]}\|\nabla\omega\|^2_{L^2(\frac{1}{2}+\frac{6}{32}\leq y\leq 5-\frac{2}{6})}
		+\nu\int_0^t\|\nabla^2\omega\|^2_{L^2(\frac{1}{2}+\frac{6}{32}\leq y\leq 5-\frac{2}{6})}ds
		\leq C\nu^6 t\big(E(t)+1\big)^5 e^{-\frac{10\eps_0}{\nu}}.
	\end{align*}
\end{lemma}

\begin{proof}
	We choose a smooth function $\eta_2(y)$ satisfying
	\begin{align*}
		\eta_2(y)=
		\left\{
		\begin{aligned}
			&1,\quad \frac{1}{2}+\frac{6}{32}\leq y\leq 5-\frac{2}{6},\\
			&0,\quad y\leq\frac{1}{2}+\frac{3}{32} \quad \text{or}\quad y\geq5-\frac{1}{6} .
		\end{aligned}
		\right.
	\end{align*}
	
	We apply $\pa_x$ on both sides of \eqref{eq: NS vorticity} and take $L^2$ inner product with $\eta_2^2\pa_x\omega$ and integrate over $0\leq s\leq t$ to have
	\begin{align*}
		\int_0^t\int_{\mathbb R^2_+}\pa_t\pa_x\omega\cdot\eta_2^2\pa_x\omega dxdyds
		+\int_0^t\int_{\mathbb R^2_+}\pa_x(U\cdot\nabla\omega)\cdot\eta_2^2\pa_x\omega dxdyds
		=\nu\int_0^t\int_{\mathbb R^2_+}\Delta\pa_x\omega\cdot\eta_2^2\pa_x\omega dxdyds.
	\end{align*}
	Integrating by parts gives rise to
	\begin{align*}
		&\frac{1}{2}\int_{\mathbb R^2_+}|\pa_x\omega(\cdot,t)|^2\eta_2^2 dxdy
		+\nu\int_0^t\|\eta_2\nabla\pa_x\omega\|^2_{L^2}ds\\
		&=-2\nu\int_0^t\int_{\mathbb R^2_+}\eta_2\pa_x\pa_y\omega\cdot\eta_2'\pa_x\omega dxdyds
		+\int_0^t\int_{\mathbb R^2_+}\eta_2\pa_x^2\omega\cdot\eta_2 U\cdot\nabla\omega dxdyds,
	\end{align*}
	which by Cauchy inequality implies
	\begin{align*}
		&\sup_{[0,t]}\|\eta_2\pa_x\omega\|^2_{L^2}
		+\nu\int_0^t \|\eta_2\nabla\pa_x\omega\|^2_{L^2}ds\\
		&\leq \frac{C}{\nu}\int_0^t\|U\cdot\nabla\omega\|^2_{L^2(\frac{1}{2}+\frac{3}{32}\leq y\leq 5-\frac{1}{6})}ds
		+C\nu\int_0^t\|\pa_x\omega\|^2_{L^2(\frac{1}{2}+\frac{3}{32}\leq y\leq 5-\frac{1}{6})}ds\\
		&\leq C\int_0^t\big(\nu^{-1}\|U\|_{L^\infty}^2+\nu\big)\|\nabla\omega\|^2_{L^2(\frac{1}{2}+\frac{3}{32}\leq y\leq 5-\frac{1}{6})}ds\\
		&\leq C\nu^{-1}\int_0^t\big(e(s)+1\big)^2\|\nabla\omega\|^2_{L^2(\frac{1}{2}+\frac{3}{32}\leq y\leq 5-\frac{1}{6})}ds\\
		&\leq C\nu^6 t\big(E(t)+1\big)^5 e^{-\frac{10\eps_0}{\nu}},
	\end{align*}
	here we used Lemma \ref{velocity estimates 1} and Lemma \ref{Sobolev est 1}. 
	
	The estimate $\|\eta_2\pa_y\omega\|_{L^2}$ can be treated in a similar way, we omit the details and conclude the proof.
\end{proof}

By the same argument, we have

\begin{lemma}\label{Sobolev est 3}
	There exists $T_0$ small enough such that for $0\leq t\leq T_0$,
	\begin{align*}
		\sup_{[0,t]}\|\nabla^2\omega\|^2_{L^2(\frac{1}{2}+\frac{9}{32}\leq y\leq 5-\frac{4}{6})}
		+\nu\int_0^t\|\nabla^3\omega\|^2_{L^2(\frac{1}{2}+\frac{9}{32}\leq y\leq 5-\frac{4}{6})}ds
		\leq C\nu^4 t\big(E(t)+1\big)^{10} e^{-\frac{10\eps_0}{\nu}},
	\end{align*}
	and
		\begin{align*}
		\sup_{[0,t]}\|\nabla^3\omega\|^2_{L^2(\frac{7}{8}\leq y\leq 4)}
		+\nu\int_0^t \|\nabla^4\omega\|^2_{L^2(\frac{7}{8}\leq y\leq 4)}ds
		\leq C\nu^2 t\big(E(t)+1\big)^{15}e^{-\frac{10\eps_0}{\nu}}.
	\end{align*}
\end{lemma}

\bigskip

\section{Estimates of the velocity via Biot-Savart law}\label{sec: Elliptic estimates}

This section is devoted to deriving several useful estimates for the velocity. 

\begin{proof}[Proof of Lemma \ref{velocity estimates 1}]

(1) We just focus on the case $i=0$. Lemma \ref{velocity formula} gives
	\begin{align*}
		u_\xi(s,y)=&-\frac{1}{2}\int_0^y e^{-|\xi|(y-z)}\big(1-e^{-2|\xi|z}\big)\omega_\xi(s,z)dz\\
		&+\frac{1}{2}\big(\int_y^{1+\mu}+\int_{1+\mu}^{+\infty}\big)e^{-|\xi|(z-y)}\big(1+e^{-2|\xi|y}\big)\omega_\xi(s,z)dz:=I_1+I_2+I_3.
	\end{align*}
Thanks to  the relation
	\begin{align}\label{weight transform 3}
		e^{\eps_0(1+\mu-y)_+|\xi|}e^{-|\xi||y-z|}
		\leq e^{\eps_0(1+\mu-z)_+|\xi|},
	\end{align}
	we have
	\begin{align*}
		e^{\eps_0(1+\mu-y)_+|\xi|}\big(|I_1|+|I_2|\big)
		\leq C\int_0^{1+\mu}e^{\eps_0(1+\mu-z)_+|\xi|}|\omega_\xi(s,z)|dz,
	\end{align*}
	which along with Lemma \ref{est of omega c} gives
	\begin{align*}
		&\left\|\sup_{0<y<1+\mu}e^{\eps_0(1+\mu-y)_+|\xi|}\big(|I_1|+|I_2|\big)\right\|_{L^1_\xi}
		\leq C\|\omega(s)\|_{Y^1_{\mu,s}}\\
		&\leq C\|\omega(s)-\omega_c(s)\|_{Y^1_{\mu,s}}+C\|\omega_c(s)\|_{Y^1_{\mu,s}}
		\leq C\big(e(s)+1\big) .
	\end{align*}
	
	A direct computation yields 
	\begin{align*}
		e^{\eps_0(1+\mu-y)_+|\xi|}|I_3|
		\leq C\int_{1+\mu}^2 |\omega_\xi(s,z)|dz
		+C\int_2^{+\infty}e^{-|\xi|/2}|\omega_\xi(s,z)|dz,
	\end{align*}
	which implies
	\begin{align*}
		&\left\|\sup_{0<y<1+\mu}e^{\eps_0(1+\mu-y)_+|\xi|}|I_3|\right\|_{L^1_\xi}
		\leq C\left\|\int_{1+\mu}^2 |\omega_\xi(s,z)|dz\right\|_{L^1_\xi}
		+C\left\|\int_2^{+\infty}e^{-|\xi|/2}|\omega_\xi(s,z)|dz\right\|_{L^1_\xi}\\
		&\leq C\int_1^2\left\|(1+|\xi|^2)^{-1/2}\right\|_{L^2_\xi}\left\|(1+|\xi|^2)^{1/2}\omega_\xi(s,z)\right\|_{L^2_\xi}dz
		+C\int_2^{+\infty}\left\|e^{-|\xi|/2}\right\|_{L^2_\xi}\|\omega_\xi(s,z)\|_{L^2_\xi}dz\\
		&\leq C\|\omega\|_{H^1(1\leq y\leq2)}
		+C\left\|e^\Psi\chi_0\psi\omega\right\|_{L^2}.
	\end{align*}
	Collecting the estimates together, we obtain the desired result. The case $i=1$ is treated similarly by replacing $\omega$ with $\pa_x\omega$, and we omit the details.\smallskip
	
	(2) Again Lemma \ref{velocity formula} gives
	\begin{align*}
		\left|\frac{v_\xi(s,y)}{y}\right|\leq& \frac{1}{2y}\int_0^y e^{-|\xi|(y-z)}\big(1-e^{-2|\xi|z}\big)|\omega_\xi(s,z)|dz\\
		&+\frac{1}{2y}\big(\int_y^{1+\mu}+\int_{1+\mu}^{+\infty}\big)e^{-|\xi|(z-y)}\big(1-e^{-2|\xi|y}\big)|\omega_\xi(s,z)|dz
		:=J_1+J_2+J_3.
	\end{align*}
	Notice that
	\begin{align*}
		\left|1-e^{-2|\xi|z}\right|\leq 2|\xi|z\leq2|\xi|y,
		\quad
		\left|1-e^{-2|\xi|y}\right|\leq 2|\xi|y,
		\quad
		\text{for}
		\quad
		z\leq y,
	\end{align*}
	which together with \eqref{weight transform 3} imply
	\begin{align*}
		e^{\eps_0(1+\mu-y)_+|\xi|}\big(|J_1|+|J_2|\big)
		\leq C\int_0^{1+\mu}e^{\eps_0(1+\mu-z)_+|\xi|}|\xi||\omega_\xi(s,z)|dz,
	\end{align*}
	which leads to
	\begin{align*}
		&\left\|\sup_{0<y<1+\mu}e^{\eps_0(1+\mu-y)_+|\xi|}\big(|J_1|+|J_2|\big)\right\|_{L^1_\xi}
		\leq C\|\pa_x\omega(s)\|_{Y^1_{\mu,s}}\\
		&\leq C\|\pa_x\omega(s)-\pa_x\omega_c(s)\|_{Y^1_{\mu,s}}
		+C\|\pa_x\omega_c(s)\|_{Y^1_{\mu,s}}
		\leq C\big(e(s)+1\big) .
	\end{align*}
	
	The term $J_3$ is treated as $I_3$ in the proof of (1):
	\begin{align*}
		\left\|\sup_{0<y<1+\mu}e^{\eps_0(1+\mu-y)_+|\xi|}|J_3|\right\|_{L^1_\xi}
		\leq C\|\omega\|_{H^1(1\leq y\leq2)}
		+C\left\|e^\Psi\chi_0\psi\omega\right\|_{L^2}.
	\end{align*}
	Thus, we derive the first inequality. The second is treated similarly by replacing $\omega$ with $\pa_x\omega$.\smallskip
	
	(3) A direct computation, together with Lemma \ref{velocity formula}, leads to
	\begin{align*}
		y\pa_y u_\xi(s,y)
		=&\frac{y}{2}\Big(\int_0^y e^{-|\xi|(y-z)}\big(1-e^{-2|\xi|z}\big)|\xi|\omega_\xi(s,z)dz\\
		&+\int_y^{+\infty}e^{-|\xi|(z-y)}\big(1+e^{-2|\xi|y}\big)|\xi|\omega_\xi(s,z)dz\\
		&-2\int_y^{+\infty}e^{-|\xi|(z-y)}e^{-2|\xi|y}|\xi|\omega_\xi(s,z)dz\Big)
		-y\omega_\xi(s,y).
	\end{align*}
	The first three terms are treated as (1) and (2). For the last term, the fundamental theorem of calculus gives rise to
	\begin{align*}
		\left\|\sup_{0<y<1+\mu}e^{\eps_0(1+\mu-y)_+|\xi|}|y\omega_\xi(s,y)|\right\|_{L^1_\xi}
		\leq C\|\omega(s)\|_{Y^1_{\mu,s}}
		+C\|y\pa_y\omega(s)\|_{Y^1_{\mu,s}}
		+C\|\pa_x\omega(s)\|_{Y^1_{\mu,s}}.
	\end{align*}
	Thus, we obtain the inequality for $y\pa_y u_\xi$. The case $y\pa_y\big(\frac{v_\xi(s)}{y}\big)$ is derived from the relation
	\begin{align*}
		y\pa_y\big(\frac{v_\xi(s)}{y}\big)
		=\pa_y v_\xi(s)-\frac{v_\xi(s)}{y}
		=-(\pa_x u)_\xi(s)-\frac{v_\xi(s)}{y}.
	\end{align*}
	
	(4) We deal with the case  $\pa_x^i u$ for $i\leq2$ firstly.
	Lemma \ref{velocity formula} gives
	\begin{align*}
		|(\pa_x^i u)_\xi(s,y)|
		\leq& \int_0^{7/8}e^{-|\xi|/8}|\xi|^i|\omega_\xi(s,z)|dz
		+\int_{7/8}^4 |\xi|^i|\omega_\xi(s,z)|dz
		+\int_4^{+\infty}e^{-|\xi|}|\xi|^i|\omega_\xi(s,z)|dz\\
		\leq& C\int_0^{7/8}|\omega_\xi(s,z)|dz
		+C\|(\pa_x^i\omega)_\xi(s,z)\|_{L^2_z(\frac{7}{8}\leq z\leq4)}
		+C \int_4^{+\infty}e^{-|\xi|/2}|\omega_\xi(s,z)|dz,
	\end{align*}
	which implies
	\begin{align*}
	&\|\pa_x^i u(s)\|_{L^\infty(1\leq y\leq3)}
		\leq\sup_{1\leq y\leq3}\left\|(\pa_x^i u)_\xi(s,y)\right\|_{L^1_\xi}\\
		&\leq C\|\omega(s)\|_{Y^1_{\mu,s}}
		+C\left\|(1+|\xi|^2)^{-\f12}\right\|_{L^2_\xi}\left\|\left\|(1+|\xi|^2)^{\f12}(\pa_x^i\omega)_\xi(s,z)\right\|_{L^2_z(\frac{7}{8}\leq z\leq4)}\right\|_{L^2_\xi}
		+C\|e^\Psi\chi_0\omega\|_{L^2}\\
		&\leq C\|\omega(s)\|_{Y^1_{\mu,s}}
		+C\|\omega(s)\|_{H^{i+1}(\frac{7}{8}\leq y\leq4)}
		+C\|e^\Psi\chi_0\psi\omega\|_{L^2}.
	\end{align*}
	
	Next we handle the case $\pa_x^i\pa_y^j u$ for $i+j\leq2$. Again Lemma \ref{velocity formula} gives
	\begin{align*}
		\pa_y u_\xi(s,y)
		=&\frac{1}{2}\Big(\int_0^y e^{-|\xi|(y-z)}\big(1-e^{-2|\xi|z}\big)|\xi|\omega_\xi(s,z)dz\\
		&+\int_y^{+\infty}e^{-|\xi|(z-y)}\big(1+e^{-2|\xi|y}\big)|\xi|\omega_\xi(s,z)dz\\
		&-2\int_y^{+\infty}e^{-|\xi|(z-y)}e^{-2|\xi|y}|\xi|\omega_\xi(s,z)dz\Big)
		-\omega_\xi(s,y).
	\end{align*}
	The first three terms are treated as the case $\pa_x^i u$. For the last term, we have
	\begin{align*}
		\sup_{1\leq y\leq3}\|\omega_\xi(s,y)\|_{L^1_\xi}
		\leq C\sup_{1\leq y\leq3}\|\omega(s,y)\|_{H^1_x}
		\leq C\|\omega(s)\|_{H^2(1\leq y\leq3)}.
	\end{align*}
	Thus, we derive that for $i\leq1$
	\begin{align*}
	&\|\pa_x^i\pa_y u\|_{L^\infty(1\leq y\leq3)}
		\leq\sup_{1\leq y\leq3}\|(\pa_x^i\pa_y^j u)_\xi(s,y)\|_{L^1_\xi}\\
		&\leq C\|\omega(s)\|_{Y^1_{\mu,s}}
		+C\|\omega(s)\|_{H^{i+2}(\frac{7}{8}\leq y\leq4)}
		+C\|e^\Psi\chi_0\psi\omega\|_{L^2}.
	\end{align*}
	
	The cases $\pa_y^2 u$ and $\pa_x^i\pa_y^j v$ are treated in a similar manner.\smallskip
	
	(5) We choose a smooth cut-off function $\eta(y)$ satisfying $\eta(y)=0$ for $y\leq 1/2$ and $\eta(y)=1$ for $y\geq\frac{3}{4}$. We decompose the velocity as
	\begin{align*}
		U=U_1+U_2:=\nabla^{\perp}\Delta_D^{-1}\big((1-\eta)\omega\big)
		+\nabla^{\perp}\Delta_D^{-1}(\eta\omega).
	\end{align*}
	
	For any $0<\mu<\mu_0$, Lemma \ref{velocity formula} leads to
	\begin{align*}
		\|U_1(s)\|_{L^\infty}
		&\leq \sup_{y>0}\|(U_1)_\xi(s,y)\|_{L^1_\xi}
		\leq \|\|\omega_\xi(s,y)\|_{L^1_y(0\leq y\leq1)}\|_{L^1_\xi}
		\leq \|\omega(s)\|_{Y^1_{\mu,s}}\\
		&\leq \|\omega(s)-\omega_c(s)\|_{Y^1_{\mu,s}}+\|\omega_c(s)\|_{Y^1_{\mu,s}} .
	\end{align*}
	For $U_2$, we utilize Hardy-Littlewood-Sobolev inequality to obtain
	\begin{align*}
		\|U_2(s)\|_{L^4}
		\leq C\|\eta\omega(s)\|_{L^{4/3}}
		\leq C\|e^\Psi\chi_0\psi\omega\|_{L^2}.
	\end{align*}
	Then by Gagliardo-Nirenberg inequality and the boundedness of the singular integral operator, we obtain
	\begin{align*}
		\|U_2(s)\|_{L^\infty}
		&\leq C\|U_2(s)\|_{L^p}^{1-2/p}\|\nabla U_2(s)\|_{L^p}^{2/p}\\
		&\leq C\|e^\Psi\chi_0\psi\omega\|_{L^2}^{1-2/p}\|\chi_0\omega\|_{L^p}^{2/p}
		\leq Ce(s).
	\end{align*}

Collecting these estimates together, we obtain the desired result.
\end{proof}

\appendix

\section{Some technical lemmas}
Here we list some technical lemmas. First of all, rewriting the Biot-Savart law, we have the following relationship between $U$ and $\om$ (see \cite{Maekawa} for the details). 

\begin{lemma}\label{velocity formula}
	Let $U=\nabla^\perp\Delta_D^{-1}\omega$ where $\om$ is defined by \eqref{eq: NS vorticity}. Then, we have
	\begin{align*}
		&u_\xi(y)=\frac{1}{2}\Big(-\int_0^y e^{-|\xi|(y-z)}\big(1-e^{-2|\xi|z}\big)\omega_\xi(z)dz
		+\int_y^{+\infty}e^{-|\xi|(z-y)}\big(1+e^{-2|\xi|y}\big)\omega_\xi(z)dz\Big),\\
		&v_\xi(y)=-\frac{i\xi}{2|\xi|}\Big(\int_0^y e^{-|\xi|(y-z)}\big(1-e^{-2|\xi|z}\big)\omega_\xi(z)dz
		+\int_y^{+\infty}e^{-|\xi|(z-y)}\big(1-e^{-2|\xi|y}\big)\omega_\xi(z)dz\Big).
	\end{align*}
\end{lemma}

\medskip

  The following lemma is used to treat the loss of derivative.
\begin{lemma}\label{analytic recovery}
	For $\widetilde \mu>\mu\geq0$, we have
\begin{align*}
		e^{\eps_0(1+\mu-y)_+|\xi|}|(\pa_x f)_\xi(y)|
		&\leq \frac{C}{\widetilde\mu-\mu}
		e^{\eps_0(1+\widetilde\mu-y)_+|\xi|}|f_\xi(y)|, 
\end{align*}
and
\begin{align*}	
		 e^{\eps_0(1+\mu)\frac{y^2}{\nu(1+t)}}
		\left|y\pa_y\Big( e^{-\frac{(y-z)^2}{4\nu(t-s)}}e^{-\nu\xi^2(t-s)} \Big)\right|&\leq\frac{C}{\sqrt{(\widetilde\mu-\mu)(t-s)}}
		e^{\eps_0(1+\widetilde\mu)\frac{y^2}{\nu(1+t)}}
		e^{-\frac{(y-z)^2}{5\nu(t-s)}}
		e^{-\nu\xi^2(t-s)}.
	\end{align*}
\end{lemma}

\begin{proof}
	The first inequality is obtained by the bound
	\begin{align*}
		(\widetilde\mu-\mu)|\xi|e^{\eps_0|\xi|\big((1+\mu-y)_+-(1+\widetilde\mu-y)_+\big)}
		\leq C.
	\end{align*}
For the second inequality, we have
	\begin{align*}
		&e^{\eps_0(1+\mu)\frac{y^2}{\nu(1+t)}}
		\left|y\pa_y\Big( e^{-\frac{(y-z)^2}{4\nu(t-s)}}e^{-\nu\xi^2(t-s)} \Big)\right|\\
		&\leq e^{\eps_0(1+\mu)\frac{y^2}{\nu(1+t)}}
		\frac{y|y-z|}{2\nu(t-s)} e^{-\frac{(y-z)^2}{4\nu(t-s)}}e^{-\nu\xi^2(t-s)}\\
		&\leq Ce^{\eps_0(1+\mu)\frac{y^2}{\nu(1+t)}}
		\frac{y}{\sqrt{\nu(t-s)}} e^{-\frac{(y-z)^2}{5\nu(t-s)}}e^{-\nu\xi^2(t-s)}\\
		&\leq Ce^{\eps_0(1+\widetilde\mu)\frac{y^2}{\nu(1+t)}}
		\frac{1}{\sqrt{(\widetilde\mu-\mu)(t-s)}} e^{-\frac{(y-z)^2}{5\nu(t-s)}}e^{-\nu\xi^2(t-s)}.
	\end{align*}
\end{proof}

The following lemma is frequently employed for handling product estimates.
\begin{lemma}\label{product estimate}
	For $0<\mu<\mu_0-\gamma s$, we have for $k=1,2$
	\begin{align*}
		\|fg\|_{Y^k_{\mu,s}}
		\leq \left\|\sup_{0<y<1+\mu}e^{\eps_0(1+\mu-y)_+|\xi|}|f_\xi(s,y)|\right\|_{L^1_\xi}\cdot\|g(s)\|_{Y^k_{\mu,s}}.
	\end{align*}
\end{lemma}

\begin{proof}
	Young inequality gives
	\begin{align*}
		&\|fg\|_{Y^k_{\mu,s}}
		\leq \left\|\int_0^{1+\mu} e^{\eps_0(1+\mu)\frac{y^2}{\nu(1+s)}}e^{\eps_0(1+\mu-y)_+|\xi|}(fg)_\xi(s,y)dy\right\|_{L^1_\xi}\\
		&\leq \Bigg\| \int_{-\infty}^{+\infty}\sup_{0<y<1+\mu}e^{\eps_0(1+\mu-y)_+|\xi-\eta|}|f_{\xi-\eta}(s,y)|\\&\quad\quad\quad\cdot\int_0^{1+\mu}e^{\eps_0(1+\mu)\frac{y^2}{\nu(1+s)}}e^{\eps_0(1+\mu-y)_+|\eta|}|g_\eta(s,y)|dyd\eta\Bigg\|_{L^1_\xi}\\
		&\leq \left\|\sup_{0<y<1+\mu}e^{\eps_0(1+\mu-y)_+|\xi|}|f_\xi(s,y)|\right\|_{L^1_\xi}\cdot \|g(s)\|_{Y^k_{\mu,s}}.
	\end{align*}
\end{proof}

The following lemma is employed to establish the uniform boundedness of  $\omega$.
\begin{lemma}\label{integral computation}
	For $\frac{1}{2}<\alpha<1$, $0<\beta<1$, $\gamma>0$ and $\mu<\mu_0-\gamma t$, it holds that
	\begin{align*}
		(\mu_0-\mu-\gamma t)^\alpha\int_0^t(\mu_0-\mu-\gamma s)^{-1-\alpha}ds&
		\leq \frac{C}{\gamma},\\
		(\mu_0-\mu-\gamma t)^\alpha\int_0^t(\mu_0-\mu-\gamma s)^{-\frac{1}{2}-\alpha}(t-s)^{-\f12}ds&\leq \frac{C}{\gamma^{\f12}},\\
		\sup_{\mu<\mu_0-\gamma t}(\mu_0-\mu-\gamma t)^\beta \ln\frac{\mu_0-\mu}{\mu_0-\mu-\gamma t}&\leq C(\gamma t)^\beta,\\
		(\mu_0-\mu-\gamma t)^\alpha\int_0^t(\mu_0-\mu-\gamma s)^{-1}s^{-1/2}ds &\leq \frac{C}{\gamma^{\f12}},
	\end{align*}
	here $C$ is a constant depending on $\mu_0$, $\alpha$ and $\beta$.
\end{lemma}

\begin{proof}
	The first inequality is quite easy, and we focus on the second one. Changing variables $t'=\gamma t$, $s'=\gamma s$ and letting $\mu'=\mu_0-\mu>t'$, we have
	\begin{align*}
		&(\mu_0-\mu-\gamma t)^\alpha\int_0^t(\mu_0-\mu-\gamma s)^{-\frac{1}{2}-\alpha}(t-s)^{-\f12}ds\\
		&=\gamma^{-\frac{1}{2}}(\mu'-t')^\alpha\int_0^{t'}(\mu'-s')^{-\frac{1}{2}-\alpha}(t'-s')^{-\f12}ds':=I.
	\end{align*}
	
	Now we let $\tilde\mu=\frac{\mu'}{t'}-1$ and $s'=t'(1-\tilde s\tilde\mu)$ to get
	\begin{align*}
		I=\gamma^{-\f12}\int_0^{1/\tilde\mu}(1+\tilde s)^{-\f12-\alpha}\tilde s^{-\f12}d\tilde s
		\leq C\gamma^{-\f12},
	\end{align*}
	here the constant $C$ is independent of $\tilde\mu$ when $\f12 <\alpha <1$.
	
	For the third inequality, we set $x=\frac{\mu_0-\mu-\gamma t}{\gamma t}\in(0,\frac{\mu_0}{\gamma t})$ and have to prove
	\begin{align*}
		\sup_{x>0}x^\beta\ln\frac{1+x}{x}\leq C,
	\end{align*}
	which is evidently true for $0<\beta<1$.
	
	For the last inequality, the change of variable $x=s\gamma, y=\frac{x}{\mu_0-\mu}$ gives
	\begin{align*}
		&(\mu_0-\mu-\gamma t)^\alpha\int_0^t(\mu_0-\mu-\gamma s)^{-1}s^{-1/2}ds 
		= (\mu_0-\mu-\gamma t)^\alpha \gamma^{-1/2}\int_0^{\gamma t}(\mu_0-\mu-x)^{-1}x^{-1/2}dx\\
		&= (\mu_0-\mu-\gamma t)^\alpha \gamma^{-1/2}(\mu_0-\mu)^{-1/2}
		\int_0^{\frac{\gamma t}{\mu_0-\mu}}(1-y)^{-1}y^{-1/2}dy\\
		&\leq C(\mu_0-\mu-\gamma t)^{\alpha-1/2} \gamma^{-1/2}
		\big(1+\ln\frac{\mu_0-\mu}{\mu_0-\mu-\gamma t}\big)
		\leq C\gamma^{-1/2},
	\end{align*}
	where we use the third inequality of this lemma and the following fact
	\begin{align*}
		\int_0^a (1-y)^{-1}y^{-1/2}dy\leq C\big(1+\ln\frac{1}{1-a}\big) \quad\text{for}\ 0<a<1.
	\end{align*}
	\end{proof}

\section{Estimates for the Euler equations}

This subsection is to derive some estimates of the Euler systems \eqref{eq: Euler}. First of all, the vorticity $\omega^e=\curl U^e$ satisfies 
\begin{align}\label{eq: Euler vorticity}
	\left\{
	\begin{aligned}
		&\pa_t\omega^e +U^e\cdot\nabla\omega^e=0,\\
		&\omega^e|_{t=0}=\omega_0 .
	\end{aligned}
	\right.
\end{align}
Since $\operatorname{supp}\omega_0\subseteq\{20\leq y\leq30\}$, there exists a $T_e>0$ such that
 \begin{align}\label{supp of omega e}
 	\operatorname{supp}\omega^e\subseteq\{10\leq y\leq40\},\quad t\in [0, T_e].
 \end{align}

\begin{proposition}\label{Euler estimates}
	There exists $T_e>0$ such that the Euler systems \eqref{eq: Euler} has a unique strong solution on $[0, T_e]$ satisfying 
	\begin{align*}
		\sup_{0\leq t\leq T_e}\sum_{\substack{i\leq 15 \\ l\leq 1}}\left\|e^{|\xi|}((1,x)\pa_t^l\pa_x^i u^e)_\xi(t,0)\right\|_{L^1_\xi\cap L^2_\xi}\leq C.
	\end{align*}
\end{proposition}

\begin{proof}
	Since $\omega_0\in L^\infty_c$, the well-posedness of the Euler system has been proved in \cite{MB}. Because of $\operatorname{supp}\omega_0\subseteq\{y\geq 20\}$, taking $T_e$ small enough, we have
	\begin{align}\label{support condition of omega}
		\operatorname{supp}\omega^e(t,x,y)\subseteq\{y\geq10\},\quad \forall \quad t\in[0, T_e].
	\end{align}
The Biot-Savart law in $\mathbb R^2_+$ gives
	\begin{align}\label{Euler velocity near the boundary}
		u^e(t,x,0)=\frac{1}{\pi}\int_{\mathbb R^2_+}\frac{\tilde y}{(x-\tilde x)^2+\tilde y^2}\omega^e(t,\tilde x,\tilde y)d\tilde xd\tilde y.
	\end{align}
	Taking Fourier transformation leads to
	\begin{align*}
		((1,x)u^e)_\xi(t,0)
		=\frac{1}{\pi}\int_{\mathbb R^2_+}(1,-2\pi \tilde y sgn\xi)e^{-2\pi i\tilde x\cdot\xi}e^{-2\pi\tilde y|\xi|}\omega^e(t,\tilde x,\tilde y)d\tilde xd\tilde y,
	\end{align*}
	which gives the case $l=0$.
	
	For $l=1$, we use the equation \eqref{eq: Euler vorticity} and integration by parts to obtain
	\begin{align*}
		(\pa_t u^e)_\xi(t,0)
		&=-\frac{1}{\pi}\int_{\mathbb R^2_+}e^{-2\pi i\tilde x\cdot\xi}e^{-2\pi\tilde y|\xi|}\big(U^e\cdot\nabla\omega^e\big)(t,\tilde x,\tilde y)d\tilde xd\tilde y\\
		&=\frac{1}{\pi}\int_{\mathbb R^2_+}\nabla\big(e^{-2\pi i\tilde x\cdot\xi}e^{-2\pi\tilde y|\xi|}\big)\cdot\big(U^e\omega^e\big)(t,\tilde x,\tilde y)d\tilde xd\tilde y,
	\end{align*}
	which gives the case $l=1$.
\end{proof}

\section*{Acknowledgement}
J. Huang is supported by NSF of China under Grant 12071492 and 12471196. C. Wang is supported by NSF of China under Grant 12071008. Z. Zhang is supported by NSF of China under Grant 12171010 and 12288101.

\end{document}